\documentclass[12pt]{amsart}
\usepackage{amssymb,epsfig,amssymb,amsthm,epic,eepic,color,amsfonts}
\usepackage{hyperref}
\usepackage[all]{xy}
\usepackage{fontenc}  
\usepackage[utf8]{inputenc}

\newtheorem{thm}{Theorem}[section]

\newtheorem{prop}[thm]{Proposition}
\newtheorem{defn}[thm]{Definition}
\newtheorem{lemme}[thm]{Lemma}
\newtheorem{cor}[thm]{Corollary}

\newtheorem{remarque}[thm]{Remark}

\newtheorem{notation}[thm]{Notation}
\newtheorem{ex}[thm]{Example}
\newtheorem{rien}[thm]{}
\numberwithin{equation}{section}

\newcommand{\be}{\begin{enumerate}}
\newcommand{\ee}{\end{enumerate}}
\newcommand{\bi}{\begin{itemize}}
\newcommand{\ei}{\end{itemize}}

\def\B{\mathbb{B}}
\def\R{\mathbb{R}}

\def\Z{\mathbb{Z}}

\def\D{\mathbb{D}}
\def\A{A_\infty}
\def\kk{\bf{k}}

\def\om{\omega}
\def\Om{\Omega}

\def\ga{\gamma}    
\def\Ga{\Gamma}

\def\al{\alpha}
\def\be{\beta}
 
\def\de{\delta}
\def\De{\Delta}
\def\vp{\varphi}

\def\la{\lambda}
\def\La{\Lambda}

\def\si{\sigma}
\def\Si{\Sigma}
\def\mor{\mathrm{Mor}}

\def\ep{\varepsilon}

\def\p{\partial}

\def\nd{\noindent}
\def\bull{\hfill$\Box$\\}
\def\proof{\nd {\bf Proof.\ }}
\def\g{{\bf g}}

\begin{document}
%\today
\vskip 1cm
\begin{center}
{\sc Morse complexes and multiplicative structures
\vspace{1cm}

Hossein Abbaspour \& Fran\c cois Laudenbach}
\end{center}

\title{}
\author{}
\address{Laboratoire de
math\'ematiques Jean Leray,  UMR 6629 du
CNRS, Facult\'e des Sciences et Techniques,
Universit\'e de Nantes, 2, rue de la
Houssini\`ere, F-44322 Nantes cedex 3,
France.}
\email{hossein.abbaspour@univ-nantes.fr, francois.laudenbach@univ-nantes.fr}



\begin{abstract} In this article we lay out the details of Fukaya's $A_\infty$-structure of the Morse complexe of a manifold possibly with boundary. We show that this $A_\infty$-structure is homotopically independent of the made choices. We emphasize the transversality arguments that make some fiber products smooth.
\end{abstract}
\maketitle
\tableofcontents
%\today
\thispagestyle{empty}
\vskip 1cm

\section{Introduction} 
In \cite{fukaya} Fukaya outlined the construction of  an $A_\infty$-category  whose 
objects are the smooth functions on a given closed manifold $M$ and the set of the morphisms 
$\mor (f,g)$ is $\Z$-module generated by the critical points of $g-f$. He describes the $\A$-operations
$$
m_n: \mor (f_1,f_2) \otimes \mor (f_1,f_2) \cdots \otimes \mor (f_{n-1},f_n) \to \mor (f_1,f_n)
$$
by counting points with sign (orientation) on the zero-dimensional moduli space of 
flow lines intersection according to the scheme provided by a generic (trivalent) rooted tree.  

As obvious as it is, these operations are only partially defined, meaning that each operation $m_n$ is only defined for 
generic  function $f_i$'s.  In particular, by taking $f_i=if$, where $f\in C^\infty (M)$ is a generic  Morse function, 
 %he suggested 
 the existence of an  $\A$-structure on the Morse complex of $f$ is suggested. 
 Note that in this example $\mor(if,(i+1)f))$ is precisely the Morse complex of $f$. 
 
 In the present article, not only we give an accurate construction of the hitherto described $\A$-structure on the Morse complex of a Morse function $f$, but also we prove that this  $\A$-structure is well-defined up to quasi-isomorphism of $\A$-algebras. It turns  out that the construction of 
$\A$-quasi-isomorphisms  requires  to extend Fukaya's $\A$-structure to manifolds with boundary.

The existence of the above-mentioned $A_\infty$-structure has been discussed by various authors (\cite{abou2, mescher} 
and more recently \cite{mazuir} ) for closed manifolds  using the gradient-tree moduli space.  Since they use  metric 
trees, the $A_\infty$-relations are the immediate consequence of breaking/gluing properties of metric trees. Another 
approach (taken in more details in \cite{char-Woodward}, for instance) is to adapt Floer-Seidel's idea (\cite{floer,seidel})
  for the construction of Lagrangian Fukaya category to the special case of  the graph  of $df$ in $T^*M$  as a 
  Lagrangian submanifold, and then translate the construction to  obtain the desired structure on the Morse complex. 

These methods, despite some advantages, rely on some sort of infinite dimension analysis for a  problem which should have  \emph{a priori} a finite dimensional solution. In this paper we propose an alternative method which uses the standard method of intersection theory \emph{ \`a la Thom} for submanifolds (with eventually conic singularities) in $M$. In order to prove  that the structure is well-defined up to $A_\infty$-quasi-isomorphisms, we are naturally led to consider the Morse theory of the manifolds with  boundary which  has already been developed  by the 
second author \cite{lauden1} for which we give %prove
 a summary.  

For a given $n$-dimensional compact manifold $M$ with boundary and a {\it generic} Morse 
function $f:M\to\R$, generic meaning that $f$ has no critical point on the boundary
and that the restriction $f_\partial$ of $f$ to the boundary $\partial M$ is a Morse function.
For the purpose of the present paper, it is useful to assume that $M$ is orientable.

We recall that there are two types $+$ and $-$ of critical points of $f_\partial$.  
A critical point $x$ of $f_\partial$ is of type $+$ (resp. $-$) % Neumann (resp. Dirichlet)
if  $\langle df(x), n(x)\rangle$ is positive (resp. negative); here $n(x)$ is a vector in $T_xM$ pointing outwards.
We shall denote by $crit_k f$ the set of critical points  
of $f$  (in the interior of $M$) of index $k$
and by $crit_k^+f_\partial$ (resp. $crit_k^-f_\partial$) the set of critical points
of $f_\partial$ of index $k\in \{0,\ldots,n-1\}$ %and 
which are of type $+$ (resp. $-$).

This setting was already considered in
\cite{lauden1} where the main idea was  to introduce so-called {\it adapted pseudo-gradients}, defined as follows.\footnote{In \cite{lauden1}, the terminology is different: the critical points of $f_\p$ of type $+$ (resp. $-$)
are said to be of Dirichlet type (resp. Neumann type). The labelling, Neumann or Dirichlet, comes from similar results which have been obtained previously 
in Witten's theory of de Rham cohomology
for manifolds with boundary (see \cite{chang, nier, shubin}). 
}
%Below, we use notation slightly different from the cited paper.

A vector field $X^+$ is said to be {\it positively adapted} to $f$ if  the  following conditions are fulfilled:
\begin{itemize}
\item[1)] $X^+\cdot f>0$ apart from $crit f \cup crit^+f_\partial$;
\item[2)] $X^+$  points inwards at every point of $\partial M$ except in some neighborhood
of $crit^+ f_\partial$ where $X^+ $ is tangent to the boundary; 
 \item[3)] near $crit f\cup crit^+ f_\partial$  the vector field  $X^+$ 
has a specific form  %(resp. $f_\partial-x_n^2$) 
with respect to the Euclidean metric of  some 
\emph{simple} Morse coordinates  (see Definition \ref{simple}). 
 \end{itemize}
 Since  the flow of $X^+$ is positively complete, each $x\in crit_k f\cup crit_{k-1}^+f_\partial$
  has a global unstable
 manifold $W^u(x)$ diffeomorphic to $\R^{n-k}$.  It has also a {\it local} stable manifold
 $W^s_{loc}(x)$ diffeomorphic to $\R^{k}$ if $x\in crit_kf$ and to the half-space $\R^k_-$ if $x\in crit^+_{k-1}f_\p$. 
  
 The vector field is said to be {\it Morse-Smale} when all these  (positively) invariant 
  manifolds  intersect mutually transversely. Under this assumption,
  after choosing arbitrarily orientations of the (local) stable manifolds, one defines a graded complex
  $$C_*(f,X^+)=C^+_*=\left( C_n^+\mathop{\longrightarrow}\limits^{\partial^+}\cdots
  C_k^+\mathop{\longrightarrow}\limits^{\partial^+}\cdots C_0^+\right).
  $$
  Here,  $C_k^+$ is the $\Z$-module freely generated by $crit_k f\cup crit_{k-1} ^+f_\partial$;  a generator $x$
  of $C_k^+$ is said to be of degree $k$; the degree of $x$ is noted $\vert x\vert$.
  The differential $\partial ^+$ is defined by choosing orientations of the local stable manifolds and
  counting with signs the connecting orbits 
  from $y$ to $x$ when  $\vert x\vert= \vert y\vert+1$ (note that the 
   unstable manifolds are co-oriented.)

  Similarly, a vector field $X^-$ is said to be {\it negatively adapted} to $f$ when it is 
  positively adapted to $-f$. Notice that $X^-\cdot f<0$  apart from $crit f\cup crit^-f_\partial$. %almost everywhere. 
  Choose
  such an $X^-$ which is Morse-Smale and choose an orientation of its unstable manifolds;
  %%the flow of $X^+$ is still positively complete.
  One defines a second complex
  $$C_*(f,X^-)=C^-_*=\left( C_n^-\mathop{\longrightarrow}\limits^{\partial^-}\cdots
  C_k^-\mathop{\longrightarrow}\limits^{\partial^-}\cdots C_0^-\right).
  $$
  Here, $C_k^-$ is the $\Z$-module freely generated by $crit_k f\cup crit_{k} ^-f_\partial$.
  Notice the shift of the grading which is justified by the equality:
  $$C_k^+(f)=C_{n-k}^-(-f).
  $$
   The differential $\partial ^-$ is defined
   on a generator $x\in C_k^-$ by counting with signs the connecting orbits of $X^-$
    from $x$ to $y\in C_{k-1}^-$. The main result in \cite{lauden1} is the following.
    
 \begin{thm} ${}$ \label{complexes}
 
 \nd {\rm 1)} The homology of the complex $C_*(f,X^-)$ is isomorphic to 
 $H_*(M;\Z)$.
 
 \nd  {\rm 2)} The homology of the complex $C_*(f,X^+)$ is isomorphic to 
 $H_*(M, \partial M;\Z)$.
\end{thm}

Now, we  present an important complement to  Theorem \ref{complexes}  
dealing with the multiplicative structures which exist on the considered complexes.

\begin{thm} \label{infty} Let $M$ be a compact  \emph{oriented} manifold. 
Then, each of the complexes $C_*^+$ and $C_*^-$
 can be endowed with a %natural 
 structure of  $A_\infty $-algebra 
$\mathcal G=\{m_1,m_2, \ldots\}$ such that $m_1$ is the differential
of the considered complex;  here $m_d$ denotes the $d$-fold product.
  
 This structure is well-defined up to ``homotopy'' from the data of
a \emph{coherent} family of  Morse-Smale approximations of $X^-$ (resp. $X^+$).
\end{thm}

The approximations in  question will be subjected to some transversality conditions for which 
the possible choices are
 not at all unique. The \emph{coherence} (Definition \ref{coherence-defn}) will be a form of naturality of these choices 
 with respect to a certain group of diffeomorphisms of $M$.

  The basic definitions about $A_\infty$-structures are recalled in Appendix \ref{appendix-Astr}. As we shall see in Section
  \ref{homotopic-structures}, the concept of \emph{homotopy} of $\A$-structures 
  is the algebraic translation of the idea of {\it cobordism} for the geometric objects we are going to introduce further.
  
  Sections \ref{towards} to \ref{orientations} are devoted to topological preparation to multiplicative structures
by means of a large use of Thom's transversality Theorem with constraints \cite{thom-mex}. Here are some more details:

-- Section \ref{gradients}  recalls from \cite{bz}
the compactification of the stable submanifolds and their $C^1$-\emph{conic} singularities. Appendix \ref{appendix-comp} states some generalities on this type of singularity.\\

-- Section \ref{trans} presents the most important tool for perturbing the stable manifolds in a \emph{coherent} way 
in Section \ref{coherence}. This is the hardest part and it relies on a new concept in transversality theory which we name \emph{immediate transversality} (see also Appendix \ref{sard}.)\\
 
-- Section \ref{towards} makes a list of transversality conditions which will be used for defining products of an $A_\infty$-structure. These conditions are generic and open. \\

-- Sections \ref{coherence} and  \ref{transition}  treat refinements on transversality conditions allowing the 
products to satisfy the $A_\infty$-relations. \\

--  Section \ref{orientations}   
deals with the orientation of the codimension-one strata in the compactified  geometric objects introduced in Section  \ref{towards}. \\

-- In Section \ref{A-infty} we introduce the $\A$-structure and prove $\A$-relations. \\

-- Section \ref{homotopic-structures} explains why different choices in the previous constructions lead to {\it concordant multi-intersections}. That is the topological ingredient for {\it homotopy} of $A_\infty$-structures.\\

The proof of Theorem \ref{infty} will be achieved in Sections \ref{A-infty} and \ref{homotopic-structures}.\\

The main example with non-empty boundary that we have in mind is 3-dimensional. Consider a link $L$ in the 3-sphere
$S^3$, equipped with the standard height function $h:S^3\to \R$.
The manifold with boundary we are interested in is $M:=S^3\smallsetminus U(L)$, where 
$U(L) $ is the interior of a small tubular neighborhood of L, built by means of an exponential map. 
%Let $f$ be the restriction of $h$ to $M$.
 In general position of $L$, the height function induces a Morse function
on $L$, and hence a generic Morse function $f$ on $M$. %The critical points of $f$ are one minimum and one maximum. 
Each maximum of $h\vert L$ gives rise to a pair of critical points of $f_\partial$,
one of type $-$ and index 2,  and one of type $+$ and index 1 (hence of degree 2 in 
$C_*^+$). Each minimum of $h\vert L$ gives rise to a pair of critical points of $f_\partial$,
one of  type $-$ and index 1,  and one of type $+$ and index 0 (hence of degree 1 in 
$C_*^+$). It is reasonable to expect that the Morse complexes of this pair $(M,f)$
informs a lot  on the topology of $L$. We have not yet explored this topic systematically.
As an exercise only, by using the {\it Massey product }
which is derived from the third product of the $A_\infty$-structure on 
the  {\it negative} complex, one could  prove {\it \`a la Morse} 
that the Borromean link is not trivial.  And this link remains non-trivial if it takes place in a ball
of any ambient 3-manifold.

More generally, one could distinguish two embeddings of a $k$-manifold into an $n$-manifold
by considering the complementary of their tubular neighborhoods and the $A_\infty$-structures of them.\\

\nd{\bf Acknowledgements.} 
We are deeply grateful to the anonymous referee who pointed out a serious gap in a previous version. 
The second author very much thanks Christian Blanchet who led him to this topic 
many years ago. We also thank Thibaut Mazuir, Ga\"el   Meigniez  and  and Tadayuki Watanabe  for  
helpful conversations.

\section{Preliminaries on adapted gradients}\label{gradients}

In this paper, we will only consider the case of the  theory {\it relative to the boundary},
dealing with the 
critical points of positive type and positively adapted gradient $X^+$. Similar results hold true for a \emph{negative-type} complex. This section of preliminaries is aimed at the following  topics:
\begin{enumerate}
\item to define the global stable manifolds;
\item to specify what are \emph{simple} Morse coordinates;\footnote{ When $M$ is closed, Harvey-Lawson \cite{harvey} named such coordinates $f$-tame.} 
\item to describe the closure of the invariant manifolds;
\item to introduce the graph of a positive semi-flow  and its compactification.
\end{enumerate}

Let $X_t^{+}$ denote the flow at time $t$ of the  positively complete vector field $X^+$. The following definition
makes sense.

 \begin{defn}  \label{global-stable} 
 For $x\in crit_k f \cup crit^+_{k-1}(f_\partial)$, 
 the global stable manifold  of $x$ with respect to $X^+$ is defined as the union
 $$W^s(x,X^+)= \mathop{\bigcup}\limits_{t>0}\left(X_t^+\right)^{-1}\left(W^s_{loc}(x)\right).$$ 
 \end{defn}

This manifold is diffeomorphic to a closed $k$-ball with punctures on the boundary corresponding to all %considered 
critical points which lie in its frontier; the points of its boundary are in $\p M$.

\begin{defn}${}$\label{simple}

\nd{\rm 1)} For $p\in crit _kf$,  \emph{simple} Morse coordinates\footnote{ So as not to confuse the coordinates and the critical point, the latter is here noted very differently.}
about $p$ are coordinates 
where $f$ reads %, up to a positive factor,
$$ f(x_1, \ldots, x_n)=f(p) -x_1^2-\cdots-x_k^2+ x_{k+1}^2+\cdots+ x_n^2\,.
$$

\nd{\rm 2)} For $p\in crit^+_{k-1}f_\p$, \emph{simple} Morse coordinates about $p$ are coordinates $(x_1,\ldots, x_n)$ 
such that $$\left\{
\begin{array}{l}
x_n= 0\quad \text{at every point of }\p M \text{ and } x_n<0 \text{ in the interior of } M;\\
f(x_1,\ldots,x_n)= f(p)-x_1^2-\cdots-x_{k-1}^2+x_k^2+\cdots+x_{n-1}^2+ x_n. % \text{ up to a positive  factor}.
\end{array}\right.
$$

\nd{\rm 3)} The vector field $X^+$ is said to be \emph{adapted} to such coordinates if, near $p\in crit^+_{k-1}f_\p$, 
it reads
$$ X^+(x_1,...,x_n)= -x_1\p_{x_1} -\cdots-x_{k-1}\p_{x_{k-1}}+x_k\p_{x_k}+\cdots+x_{n-1}\p_{x_{n-1}}-x_n\p_{x_n}
$$
\end{defn}

Then,  
 in some such simple coordinates $X^+$ is radial on each of the local 
stable/unstable manifolds.\footnote{ Saying that $X^+$ is a gradient is correct, but it is not a gradient of $f$
since it vanishes at a point where $df$ does not vanish.}
When $M$ is closed, this implies that the closure of $W^s(p)$, noted $cl(W^s(p))$, 
is a stratified set with $C^1$\emph{ conic 
singularities}  (or for short: with conic singularities):
 each stratum $\Si$ of $cl(W^s(p))$ is a smooth submanifold of $M$ %whose closure 
 and the way that $W^s(p)$ approaches $\Si$
%which transversely looks 
 looks like a cone sub-bundle---in a $C^1$ sense---of the normal disc bundle $\nu$ to $\Si$ in $M$. In each fiber
 $\nu_x$, $x\in \Si$, the trace of $W^s(p)$ is a cone based on 
 a similar submanifold %of less dimension in the unit normal 
in the unit sphere of $\nu_x$ \cite{bz}.\footnote{
As far as we know  such a claim is unknown for more general gradients.}
When the considered stratum $\Si$ is of codimension one in  $cl(W^s(p))$, the local structure of the closure of 
$W^s(p)$ is that of an \emph{open book} with finitely many pages whose $\Si$ is the \emph{binding set} (see Figure \ref{conic}). 
\begin{center}
\begin{figure}[h]
\includegraphics[scale =.6]{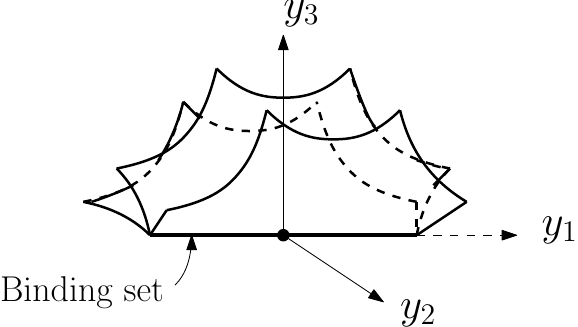}
\caption{$\dim M= 3$, $ind(p)= 2$, open book with 3 pages. The critical point $p$ is not the origin of coordinates 
$(y_1,y_2,y_3)$.}\label{conic}
\end{figure}
\end{center}
In particular, if $S$ is a submanifold of $M$ transverse to a stratum $\Si$ of $cl(W^s(p))$  %and $S\cap \Si$, 
then $S$ is transverse to $W^s(p)$ near $S\cap\Si$ (Whitney condition A).
  
  This result extends to the case with non-empty boundary under some mild  assumption.
  Here is such an assumption (Morse-Model-Transversality) which will be
  made in the rest of the paper.
  
   \begin{defn} The gradient $X^+$ is said to fulfil condition \emph{(MMT)\footnote{Acronym for Morse Model Transversality.}} if the following is satisfied:
  For every $x\in crit f\cup crit^+ f_\partial$ and $y\in crit^+ f_\partial$\,,  the 
  neighborhood $U_y$ of $y$ in $\partial M$ where $X^+$ is tangent to the boundary of $M$
  is mapped by the flow of $X^+$ transversely to $W^s(x)$.
  \end{defn}
  
 Since $X^+$ is Morse-Smale, the transversality condition  is satisfied along 
  a small neighborhood $U$ of the local unstable manifold $W^u_{loc}(y, X^+)$. 
Then, after some 
  small perturbation of $X^+$ on $U_y\smallsetminus U$ which destroyes the tangency 
  of $X^+$ to $\partial M$ over there, condition (MMT) is fulfilled. % for the pair $(y,x)$.
 Thus, condition (MMT) is generic among the 
  positively adapted vector fields. 
   The following proposition can be easily proved  by the same method as in \cite{bz}.
  
  \begin{prop}${}$ \label{frontier} It is assumed that $X^+$ is Morse-Smale and fulfils condition {\rm (MMT)}. Then 
  the following holds.
  
  \nd $1)$ The global stable manifold $W^s(x)$ is a 
  submanifold  with boundary (not closed in general).
  
  \nd $2)$ If $z\in M$ belongs to the closure of $W^s(x)$, then there exists a broken $X^+$-orbit
  from $z$ to $x$. The number of breaking critical points defines a stratification of this closure
  $cl\!\left(W^s(x)\right)$.
  
  \nd{\rm 3)} This stratification has  $C^1$ conic singularities. 
  
  \end{prop}
  
  For the remainder of this paper, we consider a generic Morse function 
$f:M\to \R$ and 
 a positively adapted gradient $X^+$. 
The transversality conditions Morse-Smale and (MMT) are assumed.

The end of this section is devoted to introduce the notion of \emph{graph of a positive semi-flow} 
$\bar X$.
This is aimed to by-pass the following difficulty: if $S$ is a submanifold of $M$ and $X$ is a gradient  
the set of points of $M$ whose positive orbit reach $S$ can be very singular. The graph will be a tool of desingularization.

\begin{defn} \label{graph}The graph $Gr(\bar X)$ of a positive semi-flow $\bar X:[0,\infty)\times M\to M$ is the  
 part of $M\times M$ made of the pairs $(x,y)$ such that $y$ belongs to the \emph{positive}
 half-orbit of $x$, that is, there exists $t\in [0,+\infty)$ such that 
 $y= \bar X(t,x).$
  If $X$ is a gradient (or has no non-constant closed orbit), this time $t$ is unique except when $x$ is a zero of $X$.
 \end{defn} 

The graph contains the diagonal of $M\times M$. For a gradient semi-flow, 
 the graph  is a non-proper $(n+1)$-dimensional submanifold,  except at the points $(a,a)$
 where $a$ is a zero of $X$. Its compactification will be 
 discussed very soon. %  in Section \ref{compact} (in particular,  Proposition \ref{strat_graph}).
 
 The first projection $M\times M\to M$ induces
 $\si: Gr(\bar X)\to M$ which is called the {\it source} map. The second projection
 induces $\tau: Gr(\bar X)\to M$ which is called the {\it target} map. 
 These two maps have  a maximal rank,  except at points $(a,a)$ with $X(a)=0$. %as above.

 \begin{ex}\label{Q} {\rm Let $Q:\R^n\to \R$ be the
 quadratic form of Morse index $k$ and rank $n$:
 $$Q(x_1,\ldots,x_n)= -x_1^2-\ldots-x_k^2+x_{k+1}^2+\ldots+x_n^2\,.
 $$ After taking local closure,
 the graph of the semi-flow  of % associated with 
 $\nabla Q$ looks  %, near the origin of $\R^n\times\R^n$,
 like, for $k= 1, \ldots,n$, the $\R$-cone  over an $n$-dimensional %(curved) 
 band (that is, $\cong \R^{n-1}\times[0,1]$)
 bounded by two affine subspaces: one is  
 $(-1,\R^{k-1}, 0, \ldots, 0)\times ( 0, \ldots, 0, \R^{n-k})\subset \R^n\times\R^n$ 
 and the other is
 the part of the diagonal over $\{x_k=-1\}$. For $k= 0$, it is similar (change $Q$ to $-Q$)---see Figure 
 \ref{singular_graph}.
 }
 \end{ex}
 
 \begin{center}
 \begin{figure}[h]
 \includegraphics[scale =.6]{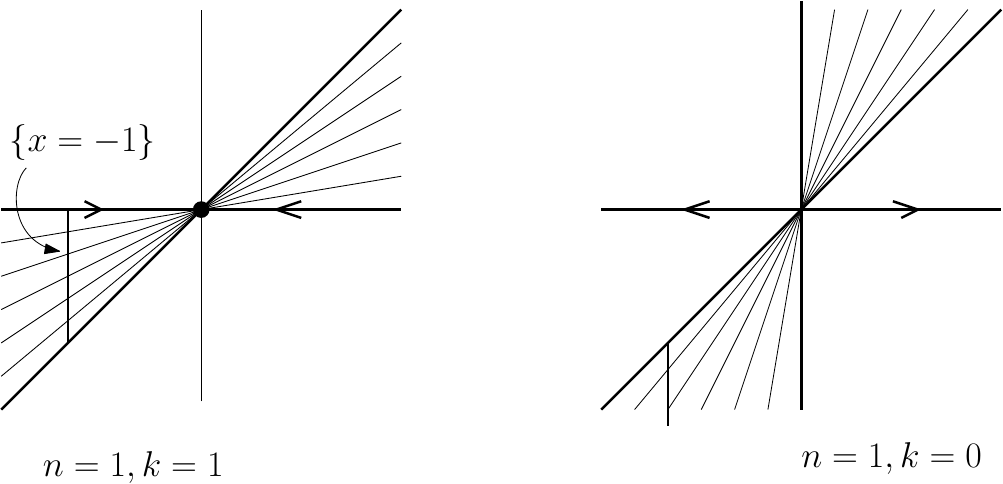}
 \caption{}\label{singular_graph}
 \end{figure}
 \end{center}

\begin{defn} Let $S$ be a submanifold of $M$, let $j:S\to M$  denote the injection 
and let $X$ be a gradient without zeroes on $S$, 
One defines the \emph{stable manifold} of $S$ with respect to X as the fiber product \footnote{ The notation as a limit in the categorical sense has the advantage 
to denote the involved maps though it is nothing but a fiber product.}  
$$W^s( S, X): = \lim\left(Gr(\bar X)\mathop{\longrightarrow}\limits^{\tau} M\mathop{\longleftarrow}
\limits^j S\right),
$$
\end{defn}
 
 In general, $W^s(S,X)$ is a singular object. But as a consequence of what we said about the rank of $\tau$, we have the following.
 
 \begin{prop}\label{tau-submer} In the above setting, assume $X$ fulfils the generic propety that 
 no zero of $X$ lies on $j(S)$. Then, the stable manifold $W^s(S,X)$ is a genuine submanifold of $Gr(\bar X) \subset M\times M$. 
 \end{prop}

 Finally, still with the same assumptions, we state something about the compactification of the graph 
 $Gr(\bar X^+)\subset M\times M$.
 %of a gradient positively adapted to $f$ as $X^+$ is so. 
 First, the diagonal of $M\times M$ and $\p M\times M$ give rise to (singular) boundary components of $Gr(\bar X^+$.
 The rest of the closure $cl(Gr):=clGr(\bar X^+)$ is described in the next proposition.

\begin{prop}\label{strat_graph}
\nd {\rm 1)} The closure $cl(Gr)$ of $Gr$ in $M\times M$ 
is made of all pairs of points $(x,y)$ where $y$ belongs to the positive orbit
of $x$ or any broken positive orbit
starting from $x$. 

\nd {\rm 2)} This $cl(Gr)$ is a stratified set. Apart from the diagonal and $\p M\times M$, the 
 strata of positive codimension 
 are  made of pairs of points $(x,y)$ where $x$  is connected  to $y$ by a broken orbit passing 
through a  non-empty sequence of critical points in $crit f\cup crit^+f_\partial$. 

\nd {\rm 3)} Among these strata,
the codimension-one strata are made of pairs of distinct points $(x,y)$
where $x$ belongs to the stable manifold
$W^s(p)$ for some $p\in crit f\cup crit^+f_\partial$ and 
$y$ belongs to the unstable manifold $W^u(p)$.\footnote{ Observe that 
the index of $p$ has no effect on the codimension of the stratum.}

\nd{\rm 4)} The singularities of $cl(Gr)$ are $C^1$-conic.
\end{prop}

\proof The proof of this proposition is very similar to the one made in \cite{bz} concerning the compactification
 of the stable/unstable manifolds of an adapted gradient. It consists---under the Morse-Smale assumption---of  looking at 
 how  the closure of a manifold with conic singularities varies when it is pushed by the  flow across a Morse model.
 The proof is the same in the case of a closed manifold, a manifold with non-empty boundary or the graph of a positive semi-flow. In the latter case, one starts from the diagonal at any point $(x,x)$ and the second $x$ is left to follow the positive semi-flow until tending to a critical point $p$. So, $(p,p)$ is a singular point of  $cl(Gr)$
 next to which two other singular strata are visible, namely $(W^s_{loc}(p)\setminus \{p\})\times \{p\}$ and 
 $\{p\}\times(W^s_{loc}(p)\setminus \{p\}) $.
 \bull

\section{Needed transversality }\label{trans}

  We start  Section \ref{trans} recalling the not very standard notion of transversality of  a finite family of smooth maps
  in the setup of sources with conic singularities. 
  Then, we specialize to  the case of the pair $\{\Si^*, \Si\}$  respectively built with the unstable and stable 
 manifolds of positive codimension of the gradient $X^+$. We construct smooth flows on $M$
      with useful properties of transversality with respect to this pair (Proposition \ref{translation-prop}). 
      And we end up this section with the so much desired \emph{skip property} in an infinite sequence
      of diffeomorphisms of $M$ close to $Id_M$. This will be 
      the main tool for getting $A_\infty$-relations from multi-intersecting 
      invariant manifolds.

\begin{defn}\label{transverse-family}
 Let $f_j: N_j\to M$,  $j\in J$, 
  be a finite set of smooth maps from 
manifolds to $M$. The family $\{f_j\}_{j\in J}$ is said to be transverse if, for  every subset 
$K\subset J$, % of cardinality $k$,
  the product map 
\begin{equation}\label{product-maps}
\prod_{j\in K}\, f_j: \prod_{j\in K} N_j\to M^{\vert K\vert}
\end{equation}
is transverse to the \emph{small diagonal} of the target. 

In that case, the fiber product %of  
$\lim_{j\in J}f_j $ %$(f_{1}, \ldots, f_{m})$ 
is said to be \emph{transversely defined}.\footnote{ 
Here and systematically in this paper, we use the notation $lim$ in the categorical sense; it has the advantage,
in comparison with the fiber product notation,
of noting the involved maps.}  This is a smooth submanifold 
of the  product  $\prod_{j\in J}N_j$.
\end{defn}

Note that in the usual definition one takes  $K=J$. In what follows, without special mention, all spaces of smooth maps
will be endowed with the $C^\infty$ topology. The same definition applies to a family of submanifolds of $M$; the maps  
 to $M$ are then meant to be the inclusions. We apply this notion to 
 submanifolds with $C^1$ conic singularities---See Appendix \ref{appendix-comp} for useful complements.
 %Here is an exercise of transversality with constraint---labeled ***  from the famous book of M. Hirsch \cite{hirsch}.
 % The following lemma was implicitely used in the proof of Proposition \ref{G-transverse}.

  Let $G= Di\!f\!f_0(M)$ denote the connected
component of $Id_M$ in the group of $C^\infty$ diffeomorphisms of $M$, equipped with the $C^\infty$ topology.
Obviously, the action of $G$ keeps $\p M$ invariant but not  pointwise fixed. 
%Let $N$ be a compact submanifold of $M$ with conic singularites transverse to $\p M$. 
 We begin with an exercise of \emph{transversality with constraint}; we solve it by following 
 Thom's idea.\footnote{ In the case of 
    no constraint, Thom gave the proof of the Transversality Theorem in \cite{thom-cob}. 
    Then he discovered that the same proof 
    applies to sections of jet spaces despite the integrability constraint \cite{thom-mex}.} 
 %---labeled ***--- from the famous book of M. Hirsch \cite{hirsch}.

\begin{prop}\label{transverse-product} In the setting of Definition \ref{transverse-family}, assume that 
$N_j$ is compact with $C^1$ conic singularities
for every $j\in J$. Let $j_0\in J$ and  $J_0:= J\smallsetminus j_0$. The family 
$\{f_j\}_{j\in J_0}$ is assumed to be transverse. Then,  for a \emph{generic} $g\in G$, 
more precisely for $g$ in some open dense subset of $G$,
the entire family is transverse if $f_{j_0}$ is replaced with $g\circ f_{j_0} $.
\end{prop}

 \proof Since transversality is an open property in $C^1$ topology and $N_{j_0}$ is compact  the set of $g$
 fulfilling the transversality requirements is open in $G$.
 We then focus on denseness.
 
 We give only a sketch of proof since the argument is classical in transversality theory.
%will be repeated  several times in the next propositions.
Let $K\subset J_0$; set $L_K:= \lim_{j\in K} f_j$ and denote $p_K: L_K\to M$ the canonical map
from the fiber product to $M$. Given $g_0\in G$, we have to show 
that the fiber product $\lim(p_K, \,g\circ f_{j_0})$ is transversely defined
for some $g$ arbitrarily close to $g_0$; actually, it is enough to consider $g_0=Id_M$.
 After this reduction, for short  we set  $f:=f_{j_0}$ 
and $N:=N_{j_0}$.

As usual for proving  a transversality theorem with constraints, it is sufficient 
 to prove that the statement holds when replacing $Id_M$ %from $Id_M\circ f_{j_0}$
 with a smooth finite dimensional family 
  in $G$ passing through $Id_M$. Indeed,
 Sard's theorem says that, if the statement holds for a smooth family \emph{in the whole}, it holds for almost every element 
 in this family.

Consider the compact set  $A:= p_K(L_K)\cap f(N)$. One covers $A$ by finitely many closed balls $\{B_i\}_{i=1}^q$ of 
$M$ equipped with Euclidean coordinates; and let $B'_i$ a larger ball concentric to $B_i$.
 For a vector $v_i$ in $\R^n$, $n=\dim M$, small enough so that the translated ball
$B_i+v_i$ remains in the interior of $B'_i$, %the same chart, 
one defines $g_{v_i}\in G$ by the formulas
\begin{equation}
g_{v_i}(x)= \left\{
\begin{array}{l}
x+v_i \text{ if } x\in B_i, \\
x\quad\quad  \text{ if } x \text{ lies outside of } B'_i, \\
\text{and some smooth interpolation in the remaining region}. %\,.
\end{array}
\right.
\end{equation} 
Denote by $(\R^n, 0)$ an arbitrarily small neighborhood of the origin in $\R^n$. % as small as desired.
Define 
 %\begin{array}{ccc}
 $\Ga:  (\R^n, 0)^{\times q}\times M\to M$ by 
 $$\Ga(v_1,...,v_q,x)= \left(g_{v_1}\circ g_{v_2}\circ\cdots\circ g_{v_q}\right)(x).
 %\end{array}
 $$  
 Its value  is equal to $g_{v_i}(x)$ when $v_k= 0$ for every $k\neq i$. 
%Such an extension is easy to find \cite{ouargla}. Anyway, it is given by the Whitney Extension Theorem \cite{whitney}(see also Appendix  A by J. Robbin in \cite{abraham}). 
    
    For every $x$ in $A$, there exists some $i\in\{1,...,q\}$ so that $x$ lies in $B_i$; here, the partial derivative
    $\left(\partial_{v_i}\Ga\right)(0,...,0,x)$ is of maximal rank $n$.
    Therefore, since $\{pt\}\times M$ is transverse to the diagonal in $M\times M$, one derives that the product map
    $$\left(p_K, \Ga\circ(Id\vert_{(\R^n,0)^q},f)\right): L_K\times \left[(\R^n, 0)^{\times q}\times N\right]\to M\times M$$
    is transverse to the diagonal of $M\times M$ (due to the $n$-dimensional 
    parameters).
    By Sard's theorem, for almost every $(v_1, ..., v_q)\in (\R^n, 0)^{\times q}$, 
    the map $\left(p_K, \Ga(v_1, \ldots, v_q,f)\right)$ is transverse to the diagonal.
    This proves the denseness part of the statement.  The genericity follows as said at the beginning of the proof. 
      ${}$\bull

    We now left generalities and focus on the concrete situation we are interested in.
    
     \begin{notation}Denote by $\Si$ (resp. $\Si^*)$
      the union of the stable (rep. unstable) manifolds of the adapted gradient $X^+$ which have a  
      positive codimension in $M$.
      \end{notation}
      The reason for not taking into account the critical points whose stable (resp. unstable)
       manifold is $n$-dimensional is that 
      transversality to them is automatic;  only transversality to their (singular) boundary is relevant. 
      
      Since $X^+$ is Morse-Smale, we know that $\Si$ has conic singularities \cite{bz}. 
      Moreover, $\Si$ is transverse to the boundary when condition (MMT) is fulfilled, 
      %except near
      just by looking at the local model.
      Note that the 0-skeleton of $\Si$ is the union of the following subsets: $crit_0f$
      %$crit_0^+f_\p$ 
      and, for every $x\in crit_0^+f_\p$, the intersection of the one-manifold
      $W^s(x, X^+)$ with $\p M$.
      
      Again, $\Si^*$  is a submanifold with conic singularities. But this time, $\Si^*$ is not transverse to $\p M$. 
      More precisely, the manifold $W^u(x, X^+)$ is tangent to $\p M$ near the critical point $x$ from which it is 
      emanating. The 0-skeleton of $\Si^*$ is the union of the following subsets: $crit_nf\cup crit^+_{n-1}f_\p$\,. 
      Observe that $\Si^{[0]}$ and $\Si^{*[0]}$ are disjoint.
      
      By the Morse-Smale property, $\Si$ is transverse to $\Si^*$, and hence, the union $\Si\cup\Si^*$ 
      is a submanifold with conic singularities (Lemma \ref{union}). That $\Si^*$ is tangent to the boundary
      will  not create any problem if  we declare that  the considered ambient isotopies %move
       are neither applied to $\Si^*$ nor $\p M$;
      only the transversality to them is preserved.
     % Moreover, $\Si\cap \Si^*$ is always transverse to the boundary.
     
     The next important definition is given in a more general setup than the pair $\{\Si^*,\Si\}$.     

           \begin{defn}\label{immediate}  Let $K\subset M$ be a submanifold wtih conic singularities transverse to $\Si$.
           A positive semi-flow $(v^t)$, 
      or its  infinitesimal generator $v$, is said to be  
      \emph{of immediate transversality to $\Si$ relative to $K$} if there exists some $\ep>0$ such that 
      $v^t(\Si)$ is transverse to  the family $\{K,\Si\}$  (or equivalently to $K\cup \Si$)
      for every $t\in (0,\ep)$.\footnote{ Rescaling the velocity allows us to take $\ep=1$.}
      \end{defn}

      In contrast to smooth submanifolds, the existence of an immediate transversality flow  is not obvious
       in presence of conic singularities. Fortunately,
      the \emph{translation flows} defined below provide us with a large family in which immediate transversality 
      is a generic property. %But this requires some more specification to the initial Morse-Smale gradient $X^+$
      %of the considered Morse function $f$ on the manifold $M$ with possibly non-empty boundary $\p M$.
       We now explain how to pass from an ``absolute'' flow of immediate transversality to a relative one.

      \begin{rien} {\sc Strata and tubes.} \label{tubes}
      The stratum $\Si_k,\ k<n,$  is the union $\cup_x W^s(x,X^+)$ for all critical points 
      $x\in crit_k\, f \cup crit^+_{k-1} f_\p$. One chooses:
      \begin{itemize}
       \item A compact domain $\underline\Si_{\,k}\subset \Si_k$ containing all critical points lying in $\Si_k$. 
      \item A compact tubular neighborhood $T_k$ of $\underline\Si_{\,k}$  (it is a trivial 
      $(n-k)$-bundle); one specifies that the fiber of $T_k$ over a critical point $x\in \Si_k$ is the local unstable manifold
      $W^u(x, X^+)$.
      \item A collar $U_k$ of the sphere bundle $ST_k$.
      \end{itemize}
      \end{rien}
      Note that the Morse Model with its so-called simple coordinates  
      and the flow of $X^+$  
      endow $T_k$ with a \emph{canonical} trivialization and each 
      fiber with a \emph{canonical} affine structure.
      These data are subject %constrained 
      to the following requirements:
      \begin{enumerate} 
      \item The union $T_0^k:= T_0\cup...\cup T_k$ is a neighborhood of the $k$-skeleton of $\Si$.\footnote{ 
      The $k$-skeleton of $\Si$ is the union  of the $j$-dimensional strata for $j\leq k$ (Definition \ref{skel}).}
      \item The boundary $\p \underline\Si_{\,k}$ is covered by $T_0^{k-1}$  and if 
       $ z\in \p \underline\Si_{\,k} \cap T_j, \ j<k$, then $ z\in U_j$
     and  $z\notin ST_j$. Moreover, if $T_{j,y}$ is the fiber of $T_j$ %over $y\in \Si_j$
      passing through $z$, then the fiber $T_{k,z}$
     is contained in an $(n-k)$-dimensional affine subspace of $T_{j,y}$.
      \item The intersection $\Si\cap T_k$ is a trivial cone sub-bundle of $T_k$ for the canonical trivialization.
           
       \end{enumerate}

      We now introduce a neighborhood of $\Si$ more manageable  than the union $T_0^{n-1}$ 
      (subsection  \ref{tubes})
      for extending  vector fields to $M$ as we have in mind. 
      
      \begin{notation}\label{prefer} The \emph{preferred neighborhood} of\, $\Si$, noted $V(\Si)$,
       is obtained from  $T_0^{n-1}$ 
      by making slits  %in it
       in the following way: one first removes  a small open \emph{exterior} collar of $ST_{n-1}$
      from $T_0^{n-1}$;
      then a small exterior collar of $ST_{n-2}$ except when it crosses $ T_{n-1}$, and so on until $ST_1$.
      {\rm (see Figure \ref{fig-preferred})}.
       \end{notation}
       Note that any germ of vector fields defined along $V(\Si)$ extends smoothly to $M$ without changing 
           the behaviour of its flow near $\Si$.             
         \begin{center}\label{fig-preferred}
 \begin{figure}[h]
 \includegraphics[scale =.5] {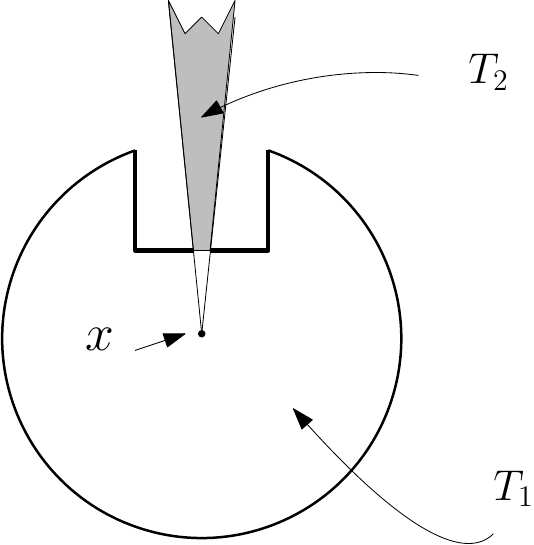} %{fig-preferred.pdf}
 \caption{For $n=3$,  a sectional view of $V(\Si)$ transverse to $\Si_1$; only one strata of $\Si_2$ adheres to 
 $\Si_1$; the sectional view of $T_2$ is in gray.}
 
 \end{figure}
 \end{center}

 \begin{defn} \label{quasi-translation}
       A germ of diffeomorphism $\vp$ is  said to be a \emph{quasi translation} if,
       in the preferred neighborhood  $V(\Si)$
       and for each tube $T_k$, it is a translation in each fiber of $T_k\cap V(\Si)$ except over a small collar of 
       $\p\underline\Si_{\,k}$, the boundary
       of the restricted $k$-stratum. Over there, %that is 
       namely on the domain of 
       reduction process (subsection \ref{re-process}), $\vp$  is the time one-map of the vector field yielded 
       by the  so-called \emph{balanced reduction
       formula} from the fiberwise translations of the tubes $T_j$, $j<k$, defining $\vp$.
        \end{defn}
        
        Note that, by construction, a quasi translation is the time-one map of a vector field (which we term alike.)
         Moreover, there are sufficiently many quasi translations so that the transversality theorem to submanifolds with 
         $C^1$ conic  singularities holds.     
      
      \begin{prop}\label{translation-prop} Let $K\subset M$ be a compact
       submanifold with $C^1$ conic singularities transverse to $\Si$.
    % Assume the gradient $X^+$ has affine transverse incidences
     % {\rm (ATI)}  (Definition  \ref{affine}). 
      Then the following holds for some real numbers  $\de>\ep'>\ep >0$:
       \begin{enumerate}
      \item %{\rm (Immediate transversality)}
       There exists a quasi translation flow $(u^t)$ which is of immediate transversality  to $\Si$: % namely,
      for every $t\in (0,\de)$, $u^t(\Si)$ is transverse to $\Si$.
         \item The generator $u$ of such a flow %Every such a flow $(u^t)$ 
         may be generically 
         $C^1$ approximated\footnote{\label{ft-generic-approx} 
         %This deals with the $C^1$ topology 
         %of germs of flows, which is coarser than the $C^1$ topology of vector fields. 
         Here, ``generically" means 
         that these approximations form a countable intersection  of dense open subsets in every neighborhood of $u$.}
        by $v$ generating  a quasi translation flow %$(v^t)$ %which is 
         of immediate transversality to the pair $\{K,\Si\}$, or equivalently, 
       for every $t\in (0,\ep')$ the image $v^t(\Si)$ is transverse to $K\cup \Si$.   
      \item{\rm (transversality-to-path)} Given such a flow $(v^t)$ and any
      $0<t_1<t_2<\ep$, the submanifold $v^{t_2}(\Si)$ is transverse to the pair
      $\{K,v^t(\Si)\}$  
        for every $t\in [0,t_1]$, that  equivalently reads \hfill \break$v^{t_2}(\Si)\pitchfork (K\cup v^t(\Si))$.   
        \end{enumerate}
      \end{prop}
      
      It is worth noting that  transversality-to-path is a very rare property. Indeed,
        an isotopy of $(\Si_t)_{t\in [0,1]}$ being given, 
      in general there is no image of $\Si_0$ transverse to every $\Si_t, \ t\in [0,1]$. 
      The third item is aimed for an iterative version of the present one (see Proposition \ref{flows-sequence}.)   \\
      
      \proof

      \nd{\bf (1)} We are going to prove this statement  by  constructing translation flows $v_k$ on each tube $T_k$   
      inductively on $k$ from $k=0$ to $k=n-1$. %up to intersecting with the preferred neighborhood $V(\Si)$. 
      Then, we will discuss  the gluing near the sphere bundle $ST_k$, more precisely inside the collar $U_k$  
      (notation from subsection \ref{tubes}.)\\

      \nd {\sc case $k=0$.} The intersection $\Si\cap ST_0$ is a compact submanifold with $C^1$ conic singularities.
      Let $C:= \Si\cap T_0$ be the cone based on it (in each arcwise component of $T_0$.)
      By Corollary \ref{trans-cones}, almost every vector $u\in\vec\R^n$ generates a   translation flow of 
      immediate transversality to $C$. 
      
      Let $u$ be such a vector. We are going to apply the \emph{reducing process} 
      which is explained in subsection \ref{re-process}
       in order to make $u $ fit all $k$-dimensional strata of $\Si$ entering $T_0$.
      Let $\Si_k$ be the $k$-dimensional stratum of $\Si$ and $\underline\Si_{\,k}$ be its selected compact sub-domain. 
      By choice of the fibered structure of $T_k$, %(ATI), 
      for $x\in\underline\Si_{\,k}\cap T_0$ the %normal 
      fiber $T_{k,x}$ with its canonical affine structure 
      is an affine subspace of $T_0$. Moreover, for $\ell>k$ and $y\in \Si_\ell \cap T_{k,x}$, the fiber  $T_{\ell,y}$ 
      is an affine subspace in $T_{k,x}$.
      For every $x\in\underline\Si_{\,k}\cap T_0$ one decomposes 
      \begin{equation}
      u= u^k_\frak h(x)+ u^k_\frak v(x)
      \end{equation}
      into its \emph{horizontal}  component $u^k_\frak h(x)$
      tangent to $\Si_k$ and its \emph{vertical} component tangent to the fiber over $x$. The same 
      decomposition is carried to each point $y$ in $T_{k,x}\cap  T_0$ by parallelism of the affine structure of  $T_0$. 
      This is the \emph{connection} induced  on $T_k$ by the affine structure of $T_0$.
      
      Let $W_k\subset \Si_k$ be a collar neighborhood of $\p\underline\Si_{\,k}$ and $E_k$
       be the part of $T_k$ over $W_k$.
      Let $\mu: W_k\to [0,1]$ be a smooth function, named the \emph{balancing function},
       equal to 1 near $\p\underline\Si_{\,k}$ and equal to 0 near the opposite side; it is lifted to $E_k$ by the 
       projection $E_k\to W_k$. The \emph{balanced reduction} of $u$ to $T_k$ is defined as follows 
       for every $y\in E_k\cap T_0$:
        \begin{equation}
       u^k_\mu(y)=\mu(y) u^k_\frak h(y) +u^k_\frak v(y). 
       \end{equation}
      By Proposition \ref{balanced}, %for a generic $u$ 
      the vector field $u^k_\frak v$ generates a translation flow of immediate 
      transversality to $\Si\cap T_k$. %Due to the associativity property \ref{assoc} 
      By subsection \ref{assox}   and the slits which have been
       made for getting
      the preferred neighborhood $V(\Si)$ from $\cup_0^{(n-1)}T_j$, the different balanced reductions appearing 
      in $T_0$ yield together a well defined vector field in $T_0\cap V(\Si)$ once $u$ is chosen. It generates a 
      quasi translation flow
      of immediate transversality to $\Si$ if $u$ does (Proposition \ref{balanced}.)\\
      
       \nd {\sc Induction step.} By abuse, 
       %\footnote{ This is mainly not to introduce a new notation for $\underline\Si_{\,k}\setminus W_k$.} 
        we neglect the domains of balanced reduction.
       One considers the tube $T_k\subset M$  trivially fibered over $\underline\Si_{\,k}$ and endowed with
       the trivial cone subbundle $\Si\cap T_k$.
        %We recall $E_k$ the part of $T_k$  over $\p\underline\Si_{\,k}$.
       
       By induction assumption, we are given a section $u^k_\p$ of the vector bundle underlying  
       $T_k$, defined near $\p\underline\Si_{\,k}$,
       and seen as  generating a translation flow in each fiber. It is
       assumed to generate a flow of immediate transversality to $\Si\cap T_k$. By Proposition \ref{cone-family},
       there is a dense open set of sections of $T_k$ over the whole $\underline\Si_{\,k}$ which generate 
       a flow of immediate transversality to $\Si\cap T_k$; the openness guaranties that some of them extend
       the germ of $u^k_\p$.
       
       To complete the induction argument, we have to apply a reduction process to every $T_\ell, \ \ell>k$. This can be 
       done by applying the reduction process, as explained when $k=0$, in each $(n-k)$-disc fiber of $T_k$. Here,
       one should specify that for getting the desired immediate transversality around a fiber which shows 
       the coplanarity phenomenon (see Corollary \ref{trans-cones})
        one has to use a slight generalization of Proposition \ref{balanced} which takes 
       into account the derivatives with respect to the base of the bundle $T_k$ (see Condition (\ref{pitch'}) below.)
       
       Note also that the balancing function attached to $\ell$ is already defined on the occasion 
       of the necessary passage of $\p\underline\Si_{\,\ell}$ in $T_0$. So, it only depends on the chosen 
       section $u^k$ of $T_k$. \\ 
               
         \nd{\bf (2)} Since transversality of $u^t(\Si)$ to $K$ holds for every $t$ in some open time interval containing 0, 
         what is missing after item (1) is the transversality of $u^t(\Si)$ to $K\cap\Si$ in some small interval $(0,\ep')$.
         One looks at this question successively in each tube $T_k$ while first neglecting the reduction 
         processes. %that is we the domain of translation flows in each fiber.
         
         In $T_0$, % first neglecting the reduction processes, that is looking at the domain of a translation flow, 
          the issue  consists of adding some more non-coplanarity conditions, namely those involving strata of $\Si$ and
         $K\cap\Si$. Here, it should be noted 
         that since $K$ is transverse to $\Si$ surely $K\cap T_0$ %$K\cap\Si\cap T_0$ 
         is not a cone (due to the vertex  of $\Si$ in each connected component of $T_0$.) But
         {\it a fortiori,} the desired requirement will be fulfilled if $K\cap (\Si\cap T_0)$ is replaced 
         with its cone in each component of $T_0$.
         So, by Corollary \ref{trans-cones} the desired transversality holds for an open dense set of translations, 
         in particular it holds for an approximation of $(u^t)$.
         
         In $T_k, \ k>0$,  one applies the same trick in each fiber: %$K\cap T_{k,x}$ %
         $K\cap (\Si\cap T_{k,x})$ 
         will be replaced with its cone
          at $x$, noted $c_xK$ for short.
          %\footnote{  The cone of $\{x\}\times 0$ is the whole fiber.}  
         By compactness of $K\cap T_k$,
          for  $x\in \underline\Si_{\,k}$ the cone $c_xK$ varies upper semi-continuously. Therefore, 
          Proposition \ref{cone-family} may be sligthly generalized even if the family $\{c_xK\}_{x\in\underline\Si_{\,k} }$
          is not a product. For completeness,  we make explicit what  replaces  condition (\ref{pitch}), that is 
           the condition for a family of translations
          $\{u^\theta(x)\}_{x\in\underline\Si_{\,k} }$ to generate a flow of immediate 
          transversality of $\Si$ to $\cup_x\bigl(c_xK\bigr)$. One of the  following two conditions has to be fulfilled.
          \begin{equation}\label{pitch'}
         \left\{
\begin{split}
&\text{ -- The translation }u^\theta(x)\text{ maps }(\Si\cap T_{k,x})\text{ transversely to }c_xK \text{ in }
\{x\}\times \B^{n-k}.\\
&\text{ -- For every hyperplane }H\text{ in }\B^{n-k}\text{ bitangent to }\Si\text{ at  }y\text{ and 
to }c_xK\text{ at }y+u^\theta(x),\\ 
&\quad\text{ the operator }\p_{\Si_k} u^\theta\vert_{x}\text{ maps the tangent space  }T_{x}\Si_k\times\{0\}\text{ transversely to the}\\
& \quad\text{ codimension-one space  }T_x\Si_k\times H.
\end{split}
\right.
         \end{equation}
         As in Proposition \ref{cone-family}, the set of translations in $T_k$ fulfilling condition (\ref{pitch'})
         is open and dense in the set of all  translations in $T_k$. So, after exhausting all tubes, the flow 
         of  $(u^t)$ from item (1) may be approximate
         to satisfy the new requirement dealing with $K$.
         
                  Over the domains of \emph{reduction process}, one needs a version of Proposition \ref{balanced}
         relative to $K$ and its fibered version based on condition (\ref{pitch'}). Its proof is similar.\footnote{ Introduce the 
         quantitative transversality
         of $T_yC$ to $T_{y+tu}(cone(K\cap C))$; and the reasoning may be led similarly.}
                
                ${}$\\
       \nd{\bf  (3)} If $K=\emptyset$, the statement follows directly from the one-parameter group formula of flows; indeed, 
       for every $t\in [0,t_1]$ the diffeomorphism $u^t$ maps $\Si$ to $u^t(\Si)$ and 
       $u^{t_2-t}(\Si)$ to $u^{t_2}(\Si)$ while carrying mutual transversality. 
       
       If $K\neq \emptyset$, this is more subtle since the isotopy of $t\mapsto v^t(\Si)\cap K$ 
       is {\it a priori} not defined by an
       autonomous flow. We are going to use the following elementary fact: 
       when $v^t(\Si) $ is both transverse to $\Si$ and $K$ one has 
        $v^t(\Si)\pitchfork ( \Si\cap K )\Leftrightarrow (v^t(\Si)\cap \Si)\pitchfork K$.
         
     The non-coplanarity conditions involving $K$, namely conditions (\ref{pitch'}), are open 
     in the $C^1$ topology. Then, a quasi translation flow $(v^t)$ being chosen  
     which satisfies (\ref{pitch'}) 
     at $t=0$ (that is, for $v^t(\Si)=\Si$)
     still satisfies it for a while. %, namely $t\in[0,\ep')$. 
     More precisely, we are knowing by item (2) that $ v^{t'}(\Si)\cap\Si$ is transverse to $K$ for every $0<t'<\ep'$. 
     By the above-mentioned openness there exists a positive $\ep <\ep'$ such that for every $0<t<t+t'<\ep $
     we have 
     $v^t(\Si)\cap v^{t+t'}(\Si) $ transverse to $K$.\,\footnote{
     In case $(v^t)$ is a translation flow the upper bound for $t+t'$ has the form  $\ep'$ minus a positive 
     linear function of the upper bound $\ep$ of $t$  (the first time where $v^t(\Si)\cap v^{t+t'}(\Si) $ 
     is not transverse to $K$.) Then it holds
     with a common upper bound of $t$ and $t+t'$. For a quasi translation flow, it is similar.}  Item (3) follows. .\bull
       
      % It is convenient to term the type of vector fields we used in the proof of Proposition \ref{translation-prop}.

      We now give an iterative version of Proposition \ref{translation-prop}. It will serve for the 
      forthcoming \emph{skip property} which is the key point to get a proof of the $A_\infty$ relations. 
     % In the remainder of  paper, for short, one  will frequently speak of generic $C^1$ approximation of vector fields instead of their generated flows.
      
      \begin{prop} \label{flows-sequence}
      
      We set $\Si^{-1}:=\Si^*$,  $\Si^0:= \Si$. 
      Then there are an infinite sequence $v_0, v_1, v_2, ...$ of vector fields which generate  quasi translation flows 
      $(v_k^t)$,
      $k=0, 1, 2 ...$, 
      and an infinite sequence of times 
      $0=t_0<t_1<t_2< ... $, 
      fulfilling the next inductive conditions  where we set 
      $\Si^{k+1}:= v_k^{t_{k+1}-t_{k}}(\Si^k)$ when $k\geq 0$: %for every non-negative integer $k$:
        \begin{enumerate}
     
      \item[$(1)_k$] For $k>0$, the vector field $v_k$ is a generic $C^1$ approximation of $v_{k-1}$.
       
       \item[$(2)_k$] For every integer $0\leq j<k$ and every $t\in [ t_j, t_{j+1}]$, the unions
       $ \Si_{-1}^{j-1}:=\Si^{-1}\cup \Si^0\cup...\cup \Si^{j-1}$ and 
       $\Si_{j+2}^{k+1}:=\Si^{j+2}\cup ...\cup \Si^{k+1}$ 
       are both transverse unions,\footnote{``transverse union'' means the family of the entries of the union 
       is a transverse family.} and the family
     $\{ \Si_{-1}^{j-1}, v_j^{t-t_j}(\Si^j),\Si_{j+2}^{k+1}\}$ is transverse. 
       
       \item[$(3)_k$]  For every %natural integer 
     $k\geq 0$ and every $t\in[ t_k, t_{k+1}]$, 
     we have $v_k^{t-t_k}(\Si^k)\pitchfork \Si_{-1}^{k-1}$. 
     Moreover $\Si^{k+1}$ is transverse to $\Si_{-1}^k$.
     
         \end{enumerate}
      \end{prop}
      
     \proof  
     The vector field $v_0$ is just the $v$ from Proposition \ref{translation-prop} which generates a flow of 
     immediate transversality to $\Si^0$ relative to $\Si^{-1}$. 
     For $t_1$ positive and small enough, $\Si^1:= v_0^{t_1}(\Si^1)$  is transverse to 
     $\Si^{-1}\cup \Si^0$ by item (2) in Proposition \ref{translation-prop}.
     %\footnote{ as $v_0^t(\Si)$, $t\in (0, t_1]$, does as well.}
     Let us explain how the vector field $v_1$ and the time $t_2$ are chosen; then the same process will be applied 
    repeatedly. 
    
    We try to continue with $v_0$ and choose a time $t_2$ so that 
    $v_0^{t-t_1}(\Si_1)$ is still transverse to  $\Si_{-1}\cup \Si_0$ for every $t\in[t_1, t_2]$. Such a time $t_2$ exists
     since  this property holds at time 
   $t_1$  and transversality to a fixed compact submanifold with conic singularities
    is  open in the $C^1$ topology. Moreover, the same holds for 
   every vector field in a small  $C^1$ neighborhood %$\mathcal N_0$ 
   of $v_0$. %Hence, $(1)_1$ follows.
    
     By  
     the transversality-to-path property
      that $v_0$ fulfills, $\Si'_2:= v_0^{t_2}(\Si^0)$ is transverse  to the family $\{\Si^{-1},v_0^t(\Si^0)\}$ for every 
    $t\in [t_0, t_1]$. Observe that this property of $\Si'_2$ is also  $C^1$ open %in the $C^1$ topology
   since $[t_0,t_1]$ is compact.  
   So, it is shared by all elements %submanifolds 
   in some open ball $\mathcal B_2$
   centered at $\Si'_2$  in the space of submanifolds of $M$ with conic singularities.
   
    Choose $v_1$ close enough to $v_0$ so that  $v_1^{t_2-t_1}$ 
     maps $\Si_1$ to an element in $\mathcal B_2$. 
    By the choice of $t_2$, we still have   $v_1^{t-t_1}(\Si^1)$ transverse to $\Si^{-1}\cup \Si^0$ for 
    every $t\in [t_1,t_2]$; and $\Si^2:= v_1^{t_2-t_1}(\Si^1)$ transverse to $\{\Si^{-1}\cup v_0^t(\Si^0)\}$
    for every $t\in [t_0,t_1]$
    
    The new
    requirement, not satisfied by $v_0$,  is that $\Si^2$ is transverse to the triple $\{\Si^{-1}, \Si^0, \Si^1\}$, 
    or equivalently,
    $\Si^2\pitchfork \left(\Si^{-1}\cup \Si^0\cup \Si^1\right)$. 
    %Thanks to what is noted right after Definition   \ref{quasi-translation}, 
    This generically holds among the approximations 
    of $v_0$ by  Proposition \ref{translation-prop} (2)---the latter being 
    applied with $K= \Si^{-1}\cup \Si^0$ and $\Si^1$ in place of $\Si$. 
    This completes the proof of the present proposition for $k=1$.

    For the induction, one notes that the properties stated in items $(2)_k$ and $(3)_k$ are 
    open with respect to all data entering them. 
    Assume $(1)_j$ -- $(3)_j$ are valid up to $j=k-1$.
    To the induction assumptions we add the existence of a decreasing sequence of open balls $\mathcal B_2\supset 
    \mathcal B_3\supset ...\supset \mathcal B_{k}$, with $\Si^j\in \mathcal B_j$,
     where every element of $\mathcal B_{k}$ (in place of $\Si^k$)
    fulfills $(2)_{k-1}$.
    
    So, $v_{k-1}$, $t_k$, $\Si^{k}$ and $\mathcal B_k$ are known; we have  
    $\Si^{k}= v_{k-1}^{t_k-t_{k-1}}(\Si^{k-1})$ which belongs to  $\mathcal B_k$. 
    One extends this flow up to a time 
    $t_{k+1}>t_k$ such that $\Si'_{k+1}:=v_{k-1}^{t_{k+1}-t_k}(\Si^k)$ still belongs to $\mathcal B_k$. 
     The ball $\mathcal B_{k+1}$ is centered at $\Si'_{k+1}$, 
    small enough for being included in $\mathcal B_k$ and
    such that each of its elements---in place of $\Si'_{k+1}$---fulfills the transversality conditions $(2)_k$.
    
     As we did when $k=1$, by  Proposition \ref{translation-prop} one may choose
      a generic  $C^1$ approximation $v_k$ of $v_{k-1}$ so that it generates 
     a  flow immediately transverse
     to $\Si^k$ relative to $\Si_{-1}^{k-1}$. In particular, $\Si^{k+1}$ is transverse to $\Si_{-1}^k$.
     One checks  conditions $(1)_k -(3)_{k}$ are valid,
     and hence, the proposition holds recursively. 
     ${}$ \bull

      \begin{defn} \label{Not3.11} 
      Let ${\bf g}_k:=( g_1, \ldots , g_{k})$  be a finite sequence 
      of $k$ elements in group $G:=Di\!f\!f_0(M)$. 
      This sequence is said to be \emph{transverse} if the family 
      %By iteration of Lemma \ref{lemma-conic}, the transversality of the family 
      $\{\Si^*, \Si,g_1(\Si),\ldots, g_{k}(\Si)\}$ is transverse in the sense of Definition \ref{transverse-family}.   
        This property will be noted  $P_k\subset G^{\times k}$.\footnote{ We  identify $G^{\times k}$
         with the set of sequences of $k$ elements in $G$.}  %a sequence in $P_k$ is said to be \emph{transverse}.
         An infinite sequence in $G$  is said to be transverse if every finite subsequence is transverse.
         \end{defn}
         
         By iteration of Lemma \ref{lemma-conic}, $P_k$ is a $C^1$ generic property.
         This is also an open property by the compactness of $\Si$ and $\Si^*$.

         \begin{defn} \label{def-skip}  A sequence $ (g_1, \ldots , g_k)\in P_k$ is said to
         have the \emph{skip} property  if for every $1\leq j\leq k$ there is given a path $(g_j^t)_{t\in [0,1]}$
         from $g_j^0=g_{j}$ to $g_j^1=g_{j-1}$ such that for every $t\in [0,1]$
         the  sequence $(g_1, \ldots, g_{j-2},g_{j}^t,g_{j+1},\ldots, g_k)$ is transverse, that is, it lies in $P_{k-1}$.
         Here, it is meant that $g_{j-1}=Id_M$ when $j=1$.
         \end{defn}
         
         One should say that the skip property is a subset $P_k^{skip}$ in $P_k^{[0,1]}$.
       %  As in Remark \ref{openness}, completed
        By the compactness of $[0,1]$,
          the skip property is $C^1$ open: it is preserved by perturbation in the $C^1$ topology   
          of the elements in $P_k$ and their associated paths.

          If a sequence has the skip property any consecutive subsequence is so since the latter has less 
          transversality requirements. Therefore, by induction on $r$ we get the following.
          
         \begin{cor}\label{Cor3.13}
          If $r$ terms are removed from a sequence $(g_1,\ldots,g_k)$ which has the skip property
        then the resulting sequence is isotopic to $(g_1, \ldots, g_{k-r})$ in $P_{k-r}$.
        \end{cor}
        
        The existence of sequences with the skip property is stated and proved below.
        
          \begin{prop}  \label{G-transverse} 
         There exists an infinite sequence $(g_0= Id_M, g_1, g_2, \ldots, g_k,\ldots)$ in $G$  
         which has the skip property.
        \end{prop}

\proof This is a direct application of Proposition \ref{flows-sequence}. The latter provides us with 
sequences of flows $(v_0)^t, ..., (v_{k-1})^t, ...$ and times $t_1, t_2, ...$. Then we set $g_1= v_0^{t_1}$, 
$g_2=v_1^{t_2-t_1}v_0^{t_1}$ 
and so forth. The required isotopy $g_j^t$ consists just to follow the flow $(v_{j-1}^t)$ backwards from the time $t_{j}$
to the time $t_{j-1}$.
%Note that it is only a proof in the $C^1$ category. But, as observed above, once the question has a solution, this one ay be approximate in the $C^\infty$ and any approximation remains a solution.\footnote{Here the topology on sequences is the \emph{weak topology}, that is, the product topology.}
\bull

%\begin{remarque}
%\end{remarque}

\section{Multi-intersections towards $A_\infty$-structures}\label{towards}
We now turn to  $A_\infty$-structures for which we refer to B. Keller \cite{keller}. In
(\cite{fukaya%,oh
}),  K. Fukaya had proposed 
 the construction of such structures on the Morse complexes of a closed manifold. We adapt his ideas to the case
 where $M$ is a manifold with a non-empty boundary. 
 
 The main point is to 
 describe multi-intersections by {\it trees}.
First, we are going to define the trees under consideration, that we name {\it Fukaya trees}.\footnote{ 
We name these trees Fukaya trees, instead of Morse trees, for two reasons. First one speaks today of the Fukaya 
Morse  theory and second we emphasize that the time of the considered flows is never involved in our approach.} 
We emphasize that no length is attached to the edges.

For us, a Fukaya tree $T\in \mathcal T_d$ is just a combinatoric object which will be used to
construct some (non-proper) submanifolds in
products of $M$ by itself $k$ times, noted $M^{\times k}$. These will be \emph{transversely defined}
(in the sense of Definition \ref{transverse-family}) and their  closure will have conic singularities.  
These manifolds will depend on the chosen {\it decoration} of $T$.
 Finally, if $\g_\infty$ is a convenient infinite sequence in $G$, a \emph{$\g_\infty$-standard} decoration of $T$  
 will allow us to define (multi)-intersection numbers.\\

%\subsection{\sc Fukaya trees.}

\begin{defn} \label{tree}Let $d$ be a positive integer. % greater than 1.  
A Fukaya tree $T$  of order $d$ %(or a $d$-tree) 
is a  finite \emph{rooted} planar tree with $d$ leaves %equipped with a total order of the leaves.
which are totally ordered.
\end{defn}

This may be thought of as an isotopy class of proper embeddings into the closed unit disc $\D$.
The end points of $T$
(the root and the leaves) lie in %embed to
$\partial \D$; the leaves are ordered clockwise in the complement of the root  in $\p \D$. 
%The remainder of $T$ embeds into %lies inthe interior of $\D$. 
 By a vertex we mean 
 an \emph{interior} vertex; it is required to have a valency greater than 2. 
An  edge is said to be \emph{interior} if  its two end points are vertices.

  Each edge is oriented from
 the root to the leaves. If $v$ is a vertex, the edges which have $v$ as origin are the \emph{branches} of $T$
 at $v$.
The edge starting from the root is named the \emph{trunk}; it is noted $e_{root}(T)$. Its upper end point will be 
noted $v_{root}(T)$.

 Let $\mathcal T_d$ be the finite set of Fukaya $d$-trees. %Inspired by the usual topology of embeddings, 
 Though there is no topology on $\mathcal T_d$, %we will say that
  a Fukaya tree will be said to be \emph{generic} if every vertex 
has valency 3; %and it is said to be 
of \emph{codimension-one} if all vertices have valency 3 except one which has valency 4. 

\begin{defn} \label{subtree}
${}$

\nd {\rm 1)} The ordered set of leaves in a Fukaya tree $T$ is denoted by $L(T)$.
Let  $T_0$ and $T_1$ be two Fukaya trees. A Fukaya embedding $j:T_0\to T_1$
is an  injective, non surjective, simplicial map which sends $L(T_0)$ to a consecutive subset of 
$L(T_1)$ increasingly.
 The image $j(T_0)$ is called a Fukaya subtree of $T_1$.
 
 \nd{\rm 2)} If $T_0$ is a Fukaya subtree of $T_1$, the Fukaya tree obtained by erasing 
 $T_0$ except its trunk is called the \emph{quotient} tree of $T_1$ by $T_0$ and is denoted by $T_1/T_0$.

\end{defn}

Topologically, $T_1/T_0$ is really a quotient since all edges  above the trunk of $T_0$ are identified to one point, 
namely $v_{root}(T_0)$. Moreover, $T_1/T_0$ is canonically a Fukaya tree. If $T_1$
is represented in $\D$
 there is a unique way up to isotopy to put $v_{root}(T_0)$ %the upper end point of $e_{root}(T_0)$
on $\p\D$ while keeping the other leaves fixed.

\begin{rien}{\bf Labelling vertices and edges.}\label{standard} 
${}${\rm

  Given a tree $T\in\mathcal T_d$,   a vertex is labelled $v_{i,j}$ if the leftmost 
 (resp. rightmost) %maximal (confusing)
 ascending monotone path starting from it in $T$
 reaches the $i$-th leaf (resp. the $j$-th leaf). If $h$ is 
the maximal number of edges in a path ascending from $v_{i,j}$ to a leaf, %$v_{i,j}$
 this vertex is said to be of \emph{generation} $h$.
 
 An edge $e\subset T$ with an origin $v$ is labelled $e_h^j(T)$ (or 
 $e_h^j$ if no possible confusion) if the following holds.
 \begin{itemize}
 \item The leftmost monotone ascending path containing $e$ and starting from $v$
 terminates in  the $j$-th leaf.
 \item  $h$ is the maximal number of edges in every 
 ascending path from $v$ containing $e$. Then, $e$ is said to be of \emph{generation} $h$.
 \end{itemize}
  }
 \end{rien}

Keeping only the edges of generation less than $h+1$ and duplicating if necessary the vertices of generation
$h$, we get a collection of disjoint Fukaya subtrees. %After some isotopy 
This can be seen as embedded in the upper half-plane, the roots being ranked in a precise order 
  on the horizontal axis. These trees form a \emph{forest of height $h$}.

 \begin{center}
 \begin{figure}[h]
 \includegraphics[scale =.6]{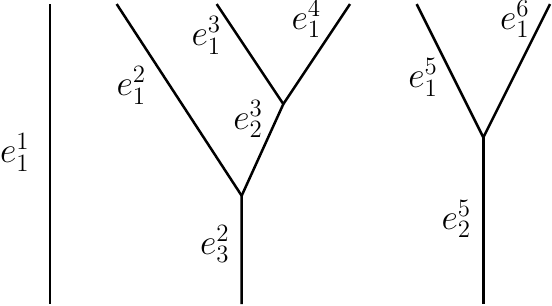}
 \caption{A height-3 forest with 6 leaves}\label{standard-fig}
 \end{figure}
 \end{center}

 \begin{defn}  \label{decoration1}  %Given the pair $(f,X^+)$, 
 Let $\g_\infty:=(g_1, g_2, ...)$ be any infinite sequence %of  $d-1$ elements 
 in $G$.
The $\g_\infty$-\emph{standard decoration} $\mathcal D_{\g_\infty}$ of a Fukaya tree $T$ 
with $d$ leaves
consists of  %associating  some vector field with each edge (interior or not). 
a collection of  $d$ vector fields of the form $X_j:=(g_{j-1})_*X^+$, $j= 1,\ldots ,d$, with $g_0=Id_M$.
The vector field  $X_j$ decorates the edge $e_h^j$, independently of $h$.
  \end{defn}

  The vector field $X_j$ is an adapted gradient of the function $f_j:=f\circ (g_{j-1})^{-1}$.
 The reason for moving the critical points, and hence $f$, is this.
 %instead of just pertubing $X^+$  while fixing the critical points is due to the following fact. 
 If the dimension of $W^s(x)$ is smaller than the half of $\dim M$ 
 there is no approximation 
 $X'$ of $X^+$ fixing  the zero $x$ 
 and putting  $W^s(x, X')$ transverse to $W^s(x, X^+)$.

 \begin{rien} \label{multi} {\bf Multi-intersection modelled on  $T$: a set theoretical construction.}
 \label{muti-settheo}
  ${}$
 
 {\rm
 We are given a \emph{generic}  Fukaya $d$-tree $T$ %with a decoration $\mathcal D$,
 and $d$ {\it entries} $(x_1,\ldots, x_d)$ where each 
 $x_i$ belongs to $crit f\cup crit^+f_\partial$,  with possible repetitions. 
 The entries decorate the leaves of $T$ clockwise. The edges are decorated by the $g_\infty$-standard decoration
 $\mathcal D_{\g_\infty}$.
  We aim to construct, by means of a precise recipe, a smooth submanifold 
  \begin{equation}
 I(T,\mathcal D_{\g_\infty}, x_1,\ldots x_d)\subset M^{\times n(T)} \text{ where } n(T)= d-1
 \end{equation}
 This set will be called the {\it multi-intersection modelled on  $T$} or the $T${\it -intersection} of the given entries
 with respect to the given decoration.
 
 Note
that $n(T)-1=d-2$ is equal to  the number of interior edges. 
The reason for this dimension will appear along the inductive construction. We first give the construction
of the $T$-intersection in the \emph{Set} category;  smoothness will be discussed later on.
}
\end{rien}

\nd{\sc Scheme of the induction.}
 It consists of associating some subset (of a certain product $M^{\times k}$) with each edge
and each vertex of $T$ in the order specified below---for brevity, neither the decoration nor the entries are noted. 
Define inductively the following subsets (depending on the chosen entries): %which are meant and not noted):
\begin{enumerate}
\item $W^s(e_1^j)\subset M$ for the edge  $e_1^j$ which ends at the $j$-th leaf of $T$, $j= 1, \ldots, d$. 
\item  $I(v)\subset M$ for every generation-one vertex; such a vertex reads $v=v_{j,j+1}$ for some $j$.
\item $W^s(e_2^j)\subset M\times M$ for every generation-2 edge $e_2^j$ of $T$.\footnote{ This item is just for the comfort of the reader.}

\item The \emph{multi-intersection}  $I(v)\subset M^{\times n(v,T)}$ for every vertex of generation $h>1$ where 
$n(v,T)-1$ is the number of interior edges of $ T$ above $v$. 

\item The \emph{stable set} $W^s(e)\subset M^{\times n(e,T)}$ for every edge of generation $h+1$
where $n(e,T)=n(v,T)+1$ if $v$ denotes the upper vertex of $e$. Finally,
 $n(e,T)-1$ is the number of interior edges above the lower vertex of $e$.
\end{enumerate}
{\it These fomulas will hold whatever the valency of the vertices.}\\

\nd{\sc Step 1.} The edge $e_1^j$  is decorated with the vector field $X_{j}=\left(g_{j-1}\right)_*X^+$ from the 
decoration $\g_\infty$-standard decoration. 
 The entry $x_j$ determines the zero $x'_j:= g_{j-1}(x_j)$ of $X_j$. One defines:
%Then, we are allowed to define the following submanifold of $M$:
\begin{equation} \label{step-1}
W^s(e_1^j):= W^s(x'_j,X_{j}),\text{ that is, the stable manifold of }x'_j\text{ with respect to } X_{j}.
\end{equation}
${}$\\

\nd{\sc Step 2.} Such a vertex  $v\in T$ is the common vertex of edges $e_1^j$ and $e_1^{j+1}$. We set 
\begin{equation}\label{step-2}
I(v):= W^s(e_1^j)\cap W^s(e_1^{j+1}),\text{ a subset of }M.
\end{equation} 
Here, $n(v,T)=1$ as announced in item (4). 
${}$\\

\nd{\sc Step 3.} Let $e_2^j$ be any edge in $T$ of generation 2, that is, $e_2^j$ is the trunk of a subtree 
of $T$ with two leaves which are numbered $j$ and $j+1$.
 Its upper vertex $v$ is interior to $T$.
By Step 2, we have $I(v)\subset M$. The edge $e_2^j$ is decorated by %  we have some gradient 
$X_j$; its flow is $\bar X_j$. One takes the 
graph of this flow $Gr(\bar X_j)\subset M\times M$ and its two maps $\si_j$ and $\tau_j$, 
respectively the source and target map. We define the {\it  stable set} $W^s(e_2^j)$ as the following fiber product 
\begin{equation}
W^s(e_2^j)= \lim\left\{%M\mathop{\longleftarrow}\limits^{p_1}
Gr(\bar X_j)\mathop{\longrightarrow}\limits^{\tau_j}M\hookleftarrow I(v)\right\}.
\end{equation}
Here,  $n(e,T)=2$ as announced.\\
\begin{center}
\begin{figure}
\includegraphics[scale=0.6]{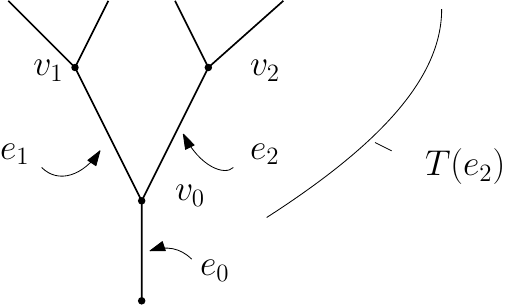}
\caption{}\label{stable}
\end{figure}
  \end{center}
\nd{\sc Step 4.} Let $v_0$ be a vertex of generation $h$. It is the origin of  two edges $e_1, e_2$ % right above $v_0$ 
ending at $v_1$ and $v_2$ respectively (Figure \ref{stable}); at least 
one of these edges is of generation $h$ and the other one is not of higher generation. 
Let $X_{j_{_1}}$ and $X_{j_{_2}}$ be their respective decorations.

For the induction, assume that for every interior edge $e$ of generation less  than $h+1$,  
%whose upper vertex is noted $v$,
the \emph{stable set} $W^s(e)$ is already defined as a subset %of the product $(n(v,T)+1)$ factors
 of $M^{\times n(e,T)}$. %where $v$ is the upper vertex of $e$. 
 In particular, $W^s(e_1)$ and $W^s(e_2)$ are subsets
in their  respective products $M^{\times n(e_1,T)}$ and $M^{\times n(e_2,T)}$.
Then, we define the \emph{multi-intersection} $I(v_0)$
by the following fiber product
\begin{equation}\label{induc-4}
I(v_0):= \lim\left\{W^s(e_1)\mathop{\longrightarrow}\limits^{\si_{j_{_1}}} M\mathop{\longleftarrow}\limits^{\ \si_{j_{_2}}}
W^s(e_2)\right\}.
\end{equation}
  %Note that the two projections $p_1$ have to be regarded as distinct since they are coming from different products of $M$.
One checks that this fiber product is contained in the product of a number of factors of $M$ equal to the total
number of interior edges above $v_0$. This is the announced formula. 

For the next step, note that the first projection 
$p_1: M^{\times n(v_0,T)}\to M$ restricted to $I(v_0)$ is just the common value of 
$\si_{j_{_1}}$ and $\si_{j_{_2}}$.\\

\nd{\sc Step 5.}\label{induc-5} Now, consider the edge $e_0$ ending at  $v_0$ from Figure \ref{stable}. It is decorated
with $X_{j_{_1}}$ by definition of a $\g_\infty$-standard decoration.
Actually, this is an arbitrary edge of generation $h+1$. As in Step (3), 
its \emph{stable set} is defined by the following fiber product 
\begin{equation}
W^s(e_0)= \lim\left\{Gr(\bar X_{j_{_1}})\mathop{\longrightarrow}\limits^{\tau_{j_{_1}}}M\mathop{\longleftarrow}
\limits^{p_1}I(v_0)\right\}.
\end{equation}
Again, the ambient product of M by itself has the announced number of factors, namely $n(e_0)= n(v_0)+1$.
This completes the induction argument
and the set theoretical construction.

In particular, we can define the $T$-intersection by
\begin{equation}
 I(T,\mathcal D_{\g_\infty}, x_1,\ldots x_d)=I(v_{root})\subset M^{\times n(T)}
 \end{equation}
That solves the problem raised in the beginning of subsection \ref{muti-settheo}, at least in the {\it Set} category.

\bull

 \begin{remarque} {\rm What we have just explained   works as well for every Fukaya trees, 
 not only the generic trees. Only the fiber product diagrams have more arrows.}
 \end{remarque}

 \begin{rien} \label{smooth}{\bf Smoothness of multi-intersections and stable sets.} ${}$
 
{\rm  Recall $\Si$ is the union of stable manifolds $W^s(x, X^+)$ where $x$ ranges over 
 $crit_* f\cup crit_{*-1}^+f_\p  $ with $*<n$. 
 And $\Si^*$ is the union of unstable manifolds $W^u(x,X^+)$ where $x$ ranges over $crit_* f\cup crit_{*-1}^+f_\p  $
 with $*>0$. In both cases, every stratum is of positive codimension.}
 \end{rien}
 
 \begin{defn} \label{requirements} A sequence $\g_d$ of $d$ elements in $G$ is said to be \emph{admissible}
 if it satisfies the following.
 \begin{enumerate}
 \item The family $\{\Si^*, \Si, g_1(\Si),\ldots, g_d(\Si)\}$ is a transverse family.
 \item For every Fukaya tree $T$ with $d+1$ leaves, decorated with the $\g_d$-standard decoration,
 and  
 for every family of edges $\{e^i\}_{i\in 1, ...,k}$ in $T$ where no pair of edges lies on a monotone arc of $T$
 the  corresponding source maps $\si_i:W^s(e_h^i)\to M$  
 form a transverse family which is  transverse to $\Si^*$.
  \end{enumerate}
  In that case, $\mathcal D_{\g_d}$ is said to be an \emph{admisible} decoration of $T$.
  The set of admissible sequences of length $d$ is noted $\mathcal A_d$.
  For $\g_{d-1}\in \mathcal A_{d-1}$, one defines $\mathcal A_d(\g_{d-1}):= \{g\in G\mid (g_1,\ldots, g_{d-1}, g)\in
   \mathcal A_d\}$.
 \end{defn}
 If such a sequence exists, which will be discussed in the next proposition, then by following 
 the induction from subsection \ref{multi},
one  inductively proves that all multi-intersections $I(v)$ and stable sets $W^s(e)$ in $T$ are 
\emph{transversely defined} in the sense of Definition \ref{transverse-family}. 
Moreover, they all are transverse to $\Si^*$.  
This is summarized in Corollary \ref{smooth-inter}.

About their compactification, the same induction, by applying the lemmas from Appendix \ref{appendix-comp}
about fiber product of manifolds with conic singularities,
tells us that these non-proper submanifolds compactify
with conic singularities.
  
  \begin{prop}\label{d-to-d+1} For every positive integer $d$
  and every  $\g_{d-1}\in \mathcal A_{d-1}$, the set $\mathcal A_d(\g_{d-1})$ is open and dense in $G$.\footnote{
   In Section \ref{coherence},  we will prove Proposition \ref{prop-coherence} which is 
    stronger than the present one.}  
  \end{prop}
  
  \proof We are going to prove this statement by induction on $d$. Let us begin with $d= 1$. It deals with the unique 
  tree with two leaves. Here, there is 
  no multi-intersection other than the usual intersection of the family $\{\Si^*, \Si, g_1(\Si)\}$.
  In other words, item (2) from Definition
  \ref{requirements} reduces to item (1). And 
  $\mathcal A_1$  is an  open dense subset of $G$ (for instance by  Proposition \ref{transverse-product}).
  
%those involved in $P_1$; that is, $\hat P_1=P_1$ and the latter is an open dense subset of $G$.
 
 Assume the statement is true for every $d'<d$ and let us prove it for $d$; this deals with the trees having  
 $d+1$ leaves. Since there are only finitely many of them, it is sufficient to give the proof for a fixed tree $T$.
 %There is a unique minimal subtree $T_0$ of $T$ which has more than one edge and contains the $(k+1)$-th leaf.
 There is a filtration of $T$ by a decreasing sequence of Fukaya subtrees 
 $T\supset T_{j_{_1}}\supset T_{j_{_2}}\supset  \cdots\supset T_{j_{_r}}\supset e_1^{d+1} $; here, $j_\ell$
 is the label of the leftmost leaf of $T_{j_{_\ell}}$.
 
 Let $L$ be the monotone path from the root to the $(d+1)$-th leaf of $T$. We have a collection of disjoint
 of subtrees $T'_1, T'_{j_{_1}}, \dots, T'_{j_{_r}}$, rooted on the successive vertices of $L$; they are 
 labelled with the label of their leftmost leaf. 
 So, their roots are the successive vertices $v_{1, d+1}, v_{j_{_1}, d+1},\ldots, v_{j_{_r}, d+1}$. 
 We do not care of the height of   $T'_{j_{_\ell}}$; so, we label its trunk only with its upper script %index
 $e^{j_{_\ell}}$. If  $\g_{d-1}$ is the sequence obtained from  $\g_{d}$ by erasing the last term, then
 the $\g_{d-1}$-standard decoration decorates all edges of $T$ except $e_1^{d+1}$.
 
 \begin{center}
 \begin{figure}[h]\label{figure7}
 \includegraphics[scale =0.6]{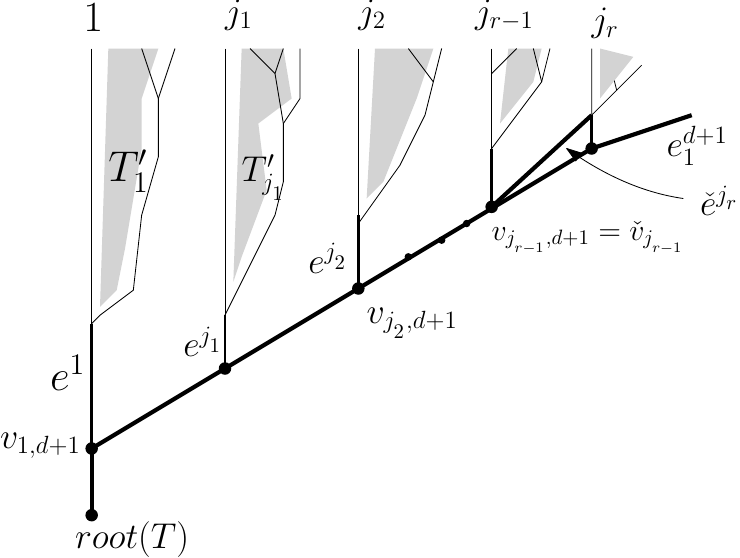}
 \caption{The thick black line represents $L\subset T$. } \label{(d+1)}
% ]\label{figure7}
 \end{figure}
 \end{center}

 Possibly, $L$ contains only one vertex, namely $v_{1, d+1}$. This case immediately reduces to the case
 $d=1$ of a tree with two leaves. We do not discuss it anymore. If $d>1$, by collapsing  the edge $e_1^{d+1}$
 to its root and ignoring  this point as a vertex one gets a new tree $\check T$ with $d$ leaves. By assumption,
 its $\g_{d-1}$-standard decoration fulfills all transversality requirements. The vertices 
 $v_{1, d+1}, v_{j_{_1}, d+1}, \ldots, v_{j_{r-1}, d+1}$ are still there, with a different %second 
 right label that we are going to neglect;
  as vertices of $\check T$
 we denote them $\check v_{1}, \check v_{j_{_1}},\ldots,  \check v_{j_{_{r-1}}}$. The right branch issued from
 $ \check v_{j_{_{r-1}}}$ in $\check T$ is a new branch whose label is $\check e^{j_r}$.
 
 We have to understand how the graft of the last branch $e_1^{d+1}$ affects the multi-intersections 
 at these vertices, that is, how we derive $I( v_{j_{_\ell}, d+1})$ from $I(\check v_{j_{_\ell}})$ for every 
 $\ell = 0,\ldots,  r-1$
 (with the convention $j_0=1$).
 And what about $I(v_{j_r, d+1})$---which does not exist in $\check T$---and
 %similarly about
  the transversality to $\Si^*$? 
 The other multi-intersections and stable 
 manifolds coming from $\check T$ are kept without any change.
 
 The manifolds $W^s(e^{j_\ell}), \ \ell= 0, \ldots, r-1,$ including $W^s(\check e^{j_r})$,
 and their source maps $\si_{j_{_\ell}}$ valued in $M$ (noted $\si$ for short)
  are transversely defined by the decoration $\g_{d-1}$-standard, 
 whatever the decoration of the last branch. This family of maps is transverse.
 Then the question is to find $g\in G$ so that  this family remains transverse when adding one particular more map,
 namely the inclusion of $g(\Si)$. 
 
 First, we prove that $I( v_{j_\ell, d+1})$ is transversely defined for a generic $g\in G$. %\mathcal G_d(\g_{d-1})$.
 We consider the diagram $\De(j_\ell)$ whose limit (or iterated fiber product) $\lim \De(j_\ell)$ is exactly 
 the definition of the multi-intersection $I(\check v_{ j_\ell})$.\\

\centerline{$\xymatrix@R=.25cm@C=.25cm{W^s(e^{j_\ell})\ar[dd]^\si&  & W^s(e^{j_{\ell+1}})\ar[dd]^\si &&&\cdots&
  W^s(e^{j_{r-1}})\ar[dd]^\si & &W^s(\check e^{j_r})\ar[dd]^\si\\
 & Gr(\bar X_{j_{\ell+1}})\ar[dl]^\si \ar[dr]_\tau & &Gr(\bar X_{j_{\ell+2}})\ar[dl]^\si \ar[dr]_\tau&&\cdots&
  &Gr(\bar X_{j_{r}})\ar[dr]_\tau\ar[dl]^\si & &\\
 M&&M&&M&\cdots& M& & M
 }$}
\vskip -.4 cm
\begin{figure}[h]
 \caption{Diagram $\De(j_\ell)$}\label{diagram}
 \end{figure}

\nd Note that the rightmost fiber product %from the right 
just produces $W^s(\check e^{j_r})$. 
 Indeed, the graph of the flow and 
 the stable manifold use the same vector field $X_{j_r}$.\footnote{ Of course, this construction of $I(\check v_{ j_\ell})$
 has an extra useless factor $M$.}
 This is a trick that allows us to graft $W^s(e_1^{d+1})$
 on $W^s(\check e^{j_r})$.
  The limit of $\De(j_\ell)$
  is equipped with %an arrow (that is, 
  a map $f_\ell: \lim \De(j_\ell)\to M$ to the rightmost $M$ in the bottom line 
 of  the diagram. In this language, the requirement  we want  is the following.
 \begin{equation} 
 \quad \text{The inclusion } g\cdot W^s(e_1^{d+1}, X^+) \hookrightarrow M\text{ is transverse to } f_\ell.
 %\text{ for every }\ell= r, ..., 1.
 \end{equation}
 By Proposition \ref{transverse-product} this property is generic in $G$. The same holds if one requires 
 the transversality to the family of all maps $f_\ell, \ell= 0,\ldots, r$. Of course, one has also to %we should  
 add all the requirements
 of item (2) in Definition \ref{requirements} involving edges which do not meet the line $L$---all of them state %mutual
 transversality of family of  source maps. 
 %Indeed, they are not affected by the presence or not of $e_1^{d+1}$. 
 The induction argument holds.
 \bull
 
 \begin{cor} \label{smooth-inter}For every admissible sequence  $\g_d$, every  
Fukaya tree $T$ with $d+1$ leaves, %with its $\mathcal \mathcal D_{\g_d}$
 and entries $(x_1, ..., x_{d+1})$, then %endowed with the $\g_d$-standard decoration 
the multi-intersection $I(T,\mathcal D_{\g_d}, x_1, ..., x_{d+1})\subset M^{\times n(T)}$ is transversely defined. Its compactification has conical singularities.
Moreover, this multi-intersection is mapped transversely to $\Si^*$ through  the first projecton 
$p_1:  M^{\times n(T)}\to M$. \bull
\end{cor}

\begin{prop}{\bf (Dimension formula)}\label{dim-formula} Given a generic Fukaya tree $T$ with $d$ leaves
endowed with an admissible decoration and given entries $(x_1, \ldots, x_d)$, we have
%$$%(1)
\begin{equation}\label{dimension}
\quad\quad \dim I(T,x_1,\ldots, x_{d})-n= \sum_{1\leq i\leq d} (\dim W^s(x_i)-n) + d-2.
\end{equation}
\end{prop}
\proof If we consider the Fukaya tree $T_0$ where all interior edges are collapsed, formula (\ref{dimension})
where $d-2$ is erased (as there is no interior edge)
reduces to the usual dimension formula for an intersection of $d$ submanifolds: it is additive
up to the shift by the ambient dimension. Each time %that
 an interior edge is created, the dimension increases by 1 since some flow is needed which generates a stable set. \bull\\

%All the above-described manifolds are oriented. The orientations will be specified in Section \ref{orientations}.They will play an important r\^ole in the $A_\infty$-structures with integral coefficients.

\begin{rien}{\bf Multi-intersection as a chain.} \label{evaluation} %${}$

{\rm In order to see the above multi-intersection
$I(T, x_1, \ldots, x_{d})$  as a chain in the Morse complex $C_*^+(f)$ whose degree  is $\dim I(T)$, 
we have to define 
the coefficient $<I(T, x_1, \ldots, x_{d}), x_{root}>$ for every test data $x_{root}\in crit f\cup crit^+f_\partial$
of degree equal to $\dim I(T)$.\footnote{ Here, the decoration is implicit and the entries are mentioned only when it seems useful for understanding.}
}
\end{rien}

We recall  the edge $e_{root}$ is decorated 
with  the vector field $X^+$. Moreover, by Corollary \ref{smooth-inter}
the projection $p_1:I(T)\to M$  to the first factor of $M^{\times n(T)}$ is transverse to 
$\Si^*$, and hence, to $W^u(x_{root}, X^+)$.
By the choice of the degree of {\it test data}, the codimension of the  unstable manifold $W^u(x_{root},X^+)$
is equal to $\dim I(T,x_1, ..., x_d)$. %Since its closure has conic singularities, 
Transversality implies 
the intersection 
$I(T, x_{root}):= p_1^{-1}\left(W^u(x_{root},X^+)\right)$ is 0-dimensional. Since its compact closure has 
conic singularities, %its frontier must be empty and hence, 
this intersection is a finite set.

Being transversely defined in an oriented manifold, $I(T)$ is oriented, once an orientation has been chosen 
for every stable manifold of critical point. In the same time, the unstable manifolds are co-oriented. 
Therefore, each point in $I(T, x_{root})$ has a sign which allows us to define $<I(T), x_{root}>$ as the algebraic 
counting of elements in this finite set.

The map $<I(T, x_1, \ldots, x_d), ->$ from test data of the right degree
 to $\Z$ will be called  the $T$-{\it evaluation map}. It depends on
  the admissible finite sequence $\g_{d-1}$ %-standard decoration of $T$.
chosen in the group $G$.

  \section{Coherence}\label{coherence}
  
    The $A_\infty$-structure that
we want to reach requires to consider all Fukaya trees, generic or not, and to decorate them
 in a {\it coherent} way. We give the precise definition right below. 
 The present section consists of mixing the skip property from Section  \ref{trans}
 and the admissibility condition (Definition \ref{requirements}.) More precisely, the issue is to prove an analogue
 of Proposition \ref{flows-sequence} in the setting of trees with admissible decorations.

 We first fix the setup for coherence. Given an infinite sequence $\g_\infty$ in the group $G= Dif\!f_0(M)$, a Fukaya tree  $T$ and a subtree $T_0$
 (Definition \ref{subtree}), then the $\g_\infty$-standard decoration of $T$ induces on $T_0$ a decoration
 $\mathcal D_T(T_0)$ which, in general, differs from
 its \emph{own}
  $\g_\infty$-standard decoration $\mathcal D_{\g_\infty}(T_0)$; the labelling of this latter is consecutive and 
  begins at $1$.
  Similarly, the quotient $T/T_0$ has also a decoration $\mathcal D(T/T_0)$ inherited from $T$ which in general is not 
  $\g_\infty$-standard; the shrinking of $T_0$ makes some gap in the decorating sequence.

 \begin{defn}{\bf (provisional)}${}$\label{coherence-prov}
 \nd{\rm 1)} Two admissible decorations of $T$ are said to be  \emph{isotopic} %homotopic
 if both lie in the same arcwise connected component of admissible decorations.\smallskip
 
 \nd{\rm 2)} A sequence $\g_d$ is said to be \emph{coherent} if it is admissible (that is, $\g_d\in \mathcal A_d$
 in the sense of Definition \ref{requirements})
 and, for every Fukaya tree $T$ with $d+1$ leaves and every subtree $T_0$, the decorations 
 $\mathcal D_T(T_0)$ and $\mathcal D(T/T_0)$ inherited from $\mathcal D_{\g_d}(T)$
 are both  isotopic to their respective own $\g_d$-standard decoration.
 
  An infinite sequence $\g_\infty$ is said to be \emph{coherent} if its finite subsequences $\g_d$ 
  are coherent for every $d$.
 \end{defn}
 
 These two examples, $\mathcal D_T(T_0)$ and $\mathcal D(T/T_0)$, are examples of \emph{pruned} trees 
 %\emph{with a gap}
 in the sense of the next definition.  
  
 \begin{defn} \label{prune} A \emph{pruned tree} is a tree with $k$ leaves and whose rightmost leaf is labelled $\ell>k$;
 that is, the labelling is not consecutive  from $1$ to $k$. It is said to be a \emph{pruned tree} with $\ell$ leaves.
 \end{defn}
 
 From this point of view, if $T_0$ is a Fukaya subtree of $T$  with no leaf labelled $1$, then
 $\mathcal D_T(T_0)$ looks as a particular case where  the pruning reads $[1,r)$. In contrast, if 1 is the label of the 
 leftmost leaf of $T_0$ the pruning is entirely made on the right of $T$ and has no effect on the decoration 
 of $T_0$; it will be said to be a \emph{useless pruning}.
 The notion of admissibility extends to the pruned trees.
  
 \begin{defn} \label{coherence-defn} A sequence $\g_d$ of diffeomorphisms of $M$ isotopic to $Id_M$ is said to be \emph{coherent} 
 if the following two conditions are fulfilled:
 \begin{enumerate}
 \item for every tree  with $d+1$ leaves, pruned or not, the (induced) $\g_d$-standard decoration is admissible;
 \item if such a tree is pruned, its induced decoration is isotopic to its own $\g_d$-standard decoration among the admissible decorations.
  \end{enumerate}
    
   An infinite sequence $\g_\infty$ is said to be \emph{coherent} if its finite subsequences $\g_d$, consecutive from $1$,
  are coherent for every $d$. 
 \end{defn}

 \begin{prop} \label{prop-coherence} There exists a coherent infinite sequence of diffeomorphisms of $M$
 whose restriction 
 to the \emph{preferred neighborhood} $V(\Si)$ of $\Si$ in $M$ is made of quasi translations. 
% More precisely, for every positive integer $d$ and every coherent sequence 
  %$\g_{d-1}$ the set $\mathcal C_d(\g_{d-1})$ is open and dense in the group $G$.
 \end{prop}
 
 \proof This will be proved by an induction on $d$  starting at $d=1$. It is somehow a combination
 of  Proposition \ref{G-transverse} about the skip property 
 %(or Proposition \ref{translation-prop} on which the latter is based \emph{via} Proposition \ref{flows-sequence}) 
 and Proposition \ref{d-to-d+1}
about admissibility. For $d=1$, there is no pruning; so, the statement reduces to transverse intersection.

As for Proposition \ref{G-transverse} we use quasi translation flows which are provided to us by
Proposition \ref{translation-prop}. Assume we have a sequence $\g_{d-1}$ whose elements are quasi translations
$g_1= v_0^{t_1}, \ldots, g_{d-1}= v_{d-2}^{(t_{d-1}-t_{d-2})}$ which form a coherent sequence of length $d-1$.
Since the transversality requirements are open conditions 
%any sequence $C^1$ close to $\g_{d-1}$ still satisfies the requirements, 
some time $t_d>t_{d-1}$ is available. 

But increasing by $1$ the number of leaves  imposes to satisfy new 
transversality requirements which make necessary to approximate the flow $ v_{d-2}$ by a suitable
$v_{d-1}$ (compare to the proof of item (2) from Proposition \ref{translation-prop}.) The new requirements in question
are essentially those resulting from diagrams like $\De(j_\ell)$ (see Figure \ref{diagram}). Here, it should be noted 
that the family of quasi translations is reach enough for providing us with finite dimensional families which are 
submersive onto the preferred neighborhood $V(\Si)$.  Therefore, the transversality theorem to a singular map is available among quasi translations. 

At this point, we have a new proof of Proposition \ref{translation-prop}). But working with flows
provides us with the so-called \emph{transversality-to-path} (item (3) from Proposition \ref{translation-prop}.)
Therefore we get the skip property in the context of Fukaya tree admissibility, which is the same as coherence 
for pruned trees with gap of length one. As in Corollary \ref{Cor3.13}, coherence for general prunings follows.
\bull

\section{Transition}\label{transition}

%Before passing to algebra, 
Proving $\A$-relations in Section \ref{A-infty} requires
%we need 
 to analyse the {\it transition} phenomenon
 from $I(T')$ to $I(T'')$,
where $T'$ and $T''$ are two ``generic'' Fukaya trees with $d$ leaves 
  on each side of 
a ``codimension-one stratum''\footnote{This is somehow abusive since no topology has been defined 
on $\mathcal T_d$; but it could have been defined.}
 in the space $\mathcal T_d$ of Fukaya trees with $d$ leaves ({\it cf.} just after
Definition \ref{tree}).\\

\begin{rien} {\bf Setting of transition}.\label{gluing}
{\rm We consider two Fukaya trees $T'$ and $T''$ which differ only in the star of $v$ 
(see Figure \ref{Figure4}). The intermediate 
 Fukaya tree $T$ has 
 exactly one vertex $v$ whose valency is 4. Up to isotopy,
$T'$ and $T''$ are the only two possible deformations from $T$ to a generic tree.
 The edge $e'$ (resp.
$e''$) is collapsed in $T'\to T$ (resp. $T''\to T$). The counting of interior edges gives $n(T')=n(T'')= n(T)+1$.

Since Proposition \ref{prop-coherence} applies to trees, generic or not, if $\g_\infty= g_1, g_2, \ldots$ is an infinite coherent sequence in the group $G$ then the sequence of vector fields 
$$X^+, g_{1*}X^+, \dots, g_{(d-1)*}X^+$$
 is a sequence of coherent $\g_{d-1}$-standard decorations common to $T'$, $T$ and  $T''$ where $\g_{d-1}$
 is the beginning subsequence of length $d$ in $\g_\infty$. In the next proposition 
 $I(T'), I(T), I(T'')$ will denote the respective multi-intersections at the vertex $v^{root}$, right above their
 common root: 
 the above decorations are implicit and the entries are arbitrary  
 critical points in $crit f\cup crit^+ f_\p$ with possible repetition.}
\end{rien}
\begin{center}
\begin{figure}[h]
\includegraphics[scale =.8]{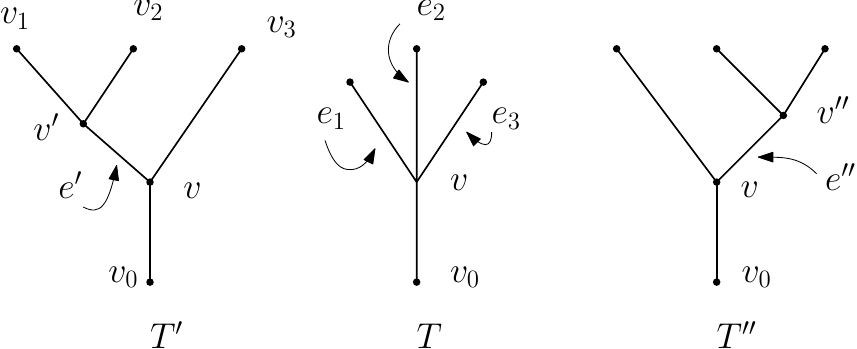}
\caption{}\label{Figure4}
\end{figure}
%${}$\hspace{-.4cm}\
\end{center}

%The following proposition will be useful in Section \ref{orientations}.

\begin{prop}\label{prop-gluing}
In this setting, the multi-intersection $I(T)$ has a natural smooth embedding 
$j':I(T)\to I(T')$, respectively $ j'':I(T)\to I(T'')$), as a boundary stratum. 
These embeddings extend to the closure $ cl\bigl(I(T)\bigr) $
in a way compatible
with the stratifications.  Then, $I(T')\mathop{\cup}\limits_{I(T)} I(T'')$
is a (piecewise smooth) manifold which is equipped with a natural  stratified compactification.
\end{prop}

Note this amalgamation is not contained in $M^{\times (n(T)+1)}$. It can only be piecewise immersed into that,
with a fold along $I(T)$.\\

\proof %{\large\color{blue} Revoir} 
It is sufficient to focus on the subtrees  $T(v_0)$, $T'(v_0)$ 
and $T''(v_0)$ rooted at $v_0$ (Figure \ref{Figure4}).
 Since $T'$ and $T''$ play the same role with respect to $T$, we look
only at $T(v_0)$ and  $T'(v_0)$. % For short, $v_0$ will named $v$.

On the one hand, the 
multi-intersection $I(v,T)$ is contained in 
$M^{\times n(v,T)}$ where $n(v,T)-1$ is  equal to  the number of {\it interior} edges of $T$ lying above $v$.
% or equivalently, $n(v,T)+2$ is equal to the number of leaves of $T(v_0)$.
We have
\begin{equation}\label{I(v,T)}
 I(v, T)=\left( W^s(e_1)\mathop{\times}\limits_MW^s(e_2)\right)\mathop{\times}\limits_M W^s(e_3)\,,
\end{equation}
where the fiber product is associative.
On the other hand, $I(v,T')$ is contained in $M^{\times n(v,T')}$ and 
the graph $G_{e'}$ of the semi-flow associated with the decoration of $e'$ is contained in  the product of the first two factors of $M^{\times n(v,T')}$.
 Thus, there is a---partially---diagonal map
 \begin{equation}\label{I(v,T')}
\left\{\begin{array}{ccc}
 \quad\quad J: M^{\times n(v,T)}&\to &M^{\times n(v,T')}\\
 (x, y ,z,\ldots)&\mapsto& (x,x, y, z, \ldots)
\end{array}\right.
\end{equation}
Observe that $I(v',T')$ is canonically isomorphic to 
$ W^s(e_1)\mathop{\times}\limits_M W^s(e_2)$ the amalgamation being made through the source map $\si$
of the respective factors---actually the projection to the first factor. Therefore, we have:
\begin{equation}\label{5.5}
 I(v,T)=
J^{-1}\left(\Bigl(G_{e'}\mathop{\times}\limits_M I(v', T')\Bigr)\mathop{\times}\limits_M W^s(e_3)\right)\, .
\end{equation}
As a consequence, $J$ induces the desired embedding $j'$, 
which is a boundary because the diagonal is a boundary of 
$G_{e'}$.  \bull

\section{Orientations}\label{orientations}

The matter of orientation is a question of Linear Algebra. Some conventions
have to be chosen.

\begin{rien}{\bf Orientation, co-orientation and boundary.} \label{or-coor}

\nd {\rm 1)} Let $E$ be a vector subspace of an
oriented vector space $V$. Let $\nu(E,V)$ be a complement to $E$ in $V$.
Then, the orientation and the co-oriention of $E$ will be related as follows:
\begin{equation}\label{or-coor1}
or\bigl(\nu(E,V)\bigr)\wedge or(E)=or(V).
\end{equation}

\nd {\rm 2)} Let $E$ be a half-space with boundary $B$. Let $\ep$ be a vector in $\nu(B,E)$
pointing outwards, where $\nu(B,E)$ is a complement to $B$ in $span(E)$.
 Then, the orientations of $B$ and $E$ will be related as follows:
\begin{equation}\label{or-coor2}
\ep\wedge or(B)= or(E).
\end{equation}
\end{rien}
When $E$ is oriented, this orientation  of $B$ is called the {\it boundary orientation};
 it is denoted by
$or_\partial (B, E)$; one also says that $B$ is  the {\it oriented boundary} of $E$.
 Notice that, when $E\subset V$, the choices 
 1) and 2) are compatible if we choose $\nu(B, V)= \nu(E,V)\oplus \ep \R$.\\

\begin{rien}{\bf Orientation and fiber product.}\label{orfiber}
 {\rm Let $E_1,E_2, V$ be three
oriented vector spaces and, for $i=1,2$, let $f_i:E_i\to V$ be a linear map.
Assume that $f_1\times f_2: E_1\times E_2\to V\times V$ is transverse to 
the diagonal $\De$. Then the fiber product $
E_{12}:=E_1\mathop{\times}\limits_V E_2$ is 
well-defined as the inverse image of $\De$ by $f_1\times f_2$. 

The first factor of $V\times V$ is seen as a  complement to $\De$ in $V\times V$. So, 
the orientation of $V$ defines a co-orientation of the diagonal.
Transversality to $\Delta$ yields a
canonical  isomorphism $\nu(E_{12},E_1\times E_2)\cong \nu(\De,V\times V)$.
Thus, $E_{12}$ is co-oriented in $E_1\times E_2$. Eventually, it is oriented according 
to (\ref{or-coor}).\\
}
\end{rien}

\begin{prop}In the case when a fiber product with three factors is defined,
the orientation is associative, that is:
%up to the factor $(-1)^{\dim V}$. Precisely: 
$\left(E_1\mathop{\times}\limits_V E_2
\right)\mathop{\times}\limits_V E_3$ and 
$E_1\mathop{\times}\limits_V\left(E_2\mathop{\times}\limits_V E_3\right)$
have the same orientation. % if $\dim V$ is even and opposite orientations if $\dim V$ is odd. %{\large ??}
\end{prop}
\proof It is sufficient to look at the small diagonal $\de_3$ in
$V\times V\times V$. In the first case it is seen as the diagonal 
of $\De\times V$ and in the second case it is seen as the diagonal
 $V\times \De$. In both cases, its co-orientation is induced by 
the orientation
of the first $V\times V$. %up to permutation of its two factors.
\bull\\

 In the  setting of subsection \ref{orfiber}, we have the following formulas.
\begin{prop} \label{pro-ori}{\rm 1)} 
Let $E_1$ be an oriented linear half-space with oriented boundary 
$B_1$ and let $E_2$ be an oriented vector space. Assume that the restriction $f_1\times f_2\vert (B_1\times E_2)$
 is transverse to $\De$. Then, the fiber product 
$ B_{12}:= B_1\mathop{\times}\limits_V E_2$ is the boundary of $E_{12}$ and 
its orientation coincides with the boundary orientation, that is: 
\begin{equation}\label{fprod-or1}
or_\partial\left(B_{12},E_{12}\right)= or\left( B_1\mathop{\times}\limits_V E_2\right)
\,.
\end{equation}

\nd {\rm 2)} Let 
 $E_2$ be now an oriented linear half-space with an oriented boundary 
$B_2$ and let $E_1$ be an oriented vector space. Assume that the restriction $f_1\times f_2\vert (E_1\times B_2)$
 is transverse to $\De$. Then, the fiber product 
$B_{12}:= E_1\mathop{\times}\limits_V B_2$ is the boundary of $E_{12}$. 
The  orientations are related as follows:
\begin{equation}\label{fprod-or2}
or_\partial\left(B_{12} ,E_{12}\right)= 
(-1)^{\dim E_1}
or\left( E_1\mathop{\times}\limits_V B_2\right)\,.
\end{equation}

\end{prop}

\proof In both cases the co-orientation of $B_{12}$ in the boundary 
$\partial\left(E_1\times E_2\right)$
 is induced by the co-orientation of the diagonal $\De$.
So, the only difference depends on the boundary orientation of
 $\partial\left(E_1\times E_2\right)$. In the first case, 
the boundary orientation is the product orientation 
$or_\partial(B_1, E_1)\wedge or(E_2)$. In the second case, we have:
$$or_\partial\left(\partial(E_1\times E_2),E_1\times E_2 \right)
=(-1)^{\dim E_1}or(E_1)\wedge or_\partial(B_2, E_2)\,.
$$\bull

%\begin{remarque}{\rm
 Of course, all of that was previously  said in the linear case applies
word to word
in the non-linear case to fiber products of manifolds with boundary 
when they are defined, that is, under some transversality assumptions.
The intersection of two transverse submanifolds is a particular case of 
the previous discussion.\\
%\end{remarque}

%\begin{rien}
\nd{\bf  Orientation and graph of a semi-flow.} 
Let $e$ be an edge (interior or not) in a decorated
Fukaya tree $(T,\mathcal D)$ and let $X_e$ be the gradient decorating $e$.
Let $G_e$ be the graph of its positive semi-flow $\bar X_e$.
The source map $\si_e$ makes  $G_e$    a $[0, +\infty)$-bundle
over $M$. By convention,  $G_e$ will be oriented like $or([0,\infty))\wedge or(M)$. Recall also the target map
$\tau_e: G_e\to M, \ (t,x)\mapsto \bar X_e^t(x)$.

\begin{prop}\label{comp-graph}
Let $z\in crit f\cup crit^+f_\partial$ and let $H$  be the codimension-one 
stratum %that $z$ generates by breaking orbits 
in the closure of $G_e\subset M\times M$ made of orbits that are broken at $z$.\footnote{ See Proposition \ref{strat_graph} item 3).}
Denote by $\dot W^s(z)$ the stable manifold punctured at $z$; and similarly for 
the unstable manifold $\dot W^u(z)$.
Then we have 
\begin{equation}\label{or-graph}
H=  \dot W^s(z)\times\dot W^u(z)
\end{equation}
as oriented manifolds where $H$ is oriented as a boundary component
of $G_e$. 
Moreover, the right handside of (\ref{or-graph}) is a sub-product of $M\times M$.

\end{prop}

\nd {\bf Proof.} First, recall that $W^u(z)$ is oriented arbitrarily; it is also  co-oriented
so that \break$co$-$or(W^u(z))\wedge or(W^u(z))= or(M)$. By convention %(\ref{or-stable}), 
the stable manifold 
is oriented by the co-orientation 
of the unstable manifold. Thus, the right hand side of (\ref{or-graph})
has the orientation of $M$.

Now, take a pair $(x,y)\in \dot  W^s(z)\times\dot W^u(z)$ and a small %positive 
$\ep>0$.
Set $a= x+\ep\vec{z y}$ in the affine structure of the Morse model  about $z$. The  orbit of
$a$ intersects the affine line $y+\R \vec{zx}$ in exactly one point $a'$ at some  time 
$t'$; we have $\bar X^{t'}_e(a)=a'$. So, for some small enough $\de$ and $0<\ep<\de$,
 we have a {\it collar map} 
$$ \begin{array}{ccc}
C: (0, \de) \times\dot W^s(z)\times\dot  W^u(z) &\longrightarrow& G_e\\
(\ep,x,y)&\mapsto & (a,a')
\end{array}
$$
which extends to a diffeomorphism  $\{0\}\times \dot W^s(z)\times\dot W^u(z)\cong H$.
By a computation in the Morse model, it is seen that, fixing $\ep$,
the map $\tau_e \circ C_{\ep}: (x,y)\mapsto a'$ is orientation preserving. %direct. 
Moreover, making  $\ep$ decrease
(which is the outgoing direction along the boundary)
makes $t'$ increase. %(which is the positive direction of the fiber of $\si_e$).
Altogether, we have the desired isomorphism of orientations.
\bull\\

\nd{\bf Orientation and multi-intersection.}

Let $T$ be a generic 
tree with $d$ leaves, an admissible decoration $\mathcal D$ and entries $x_1,...,x_d$. Here we consider the case
where  $e$ is the \emph{trunk} of 
$T$ (see right after Definition \ref{tree}).
Suppose 
$z\in crit _*f\cup crit_{*-1}^+ f_\partial$ has a degree %an index 
(that is its value of $*= \dim W^s(z)$)  which is equal 
to $\dim I(T, \mathcal D, x_1,...,x_d)$. Let  $\si_e: G_e\to M$ be the source map.

Then $z$ determines a codimension-one stratum $H$ (possibly empty), made of orbits broken at $z$, in the closure of 
$W^s(e)\subset M^{\times n(e)}$. This $H$ is mapped by $\si_e$ transversely to the unstable manifold $W^u(z)$
since the decoration is admissible. By the dimension assumption, the intersection 
$\si_e^{-1}(W^u(z))\cap H$ is made of a finitely many signed points.

\begin{defn}\label{multiplicity}
The sum $\mu(z)$ of the above signs is named the \emph{(algebraic) multiplicity} of $H$ as a boundary component of 
$W^s(e)$. It is also  the coefficient of $z$ in the chain represented by $I(T, x_1, ... ,x_d)$
(an admissible decoration being implicit---see subsection \ref{evaluation}).
\end{defn}

Finally, let $T_0$ be a strict sub-tree of $T$ with $k$ leaves, 
let $e$ be the trunk of $T_0$, an interior edge of $T$. This time,
$z$ is assumed to generate a codimension-one stratum $H$ in the closure of 
$W^s(e, x_{j+1}, ..., x_{j+k})$ where $j$ is the label of the leaf which lies just to the left of 
the leaves of $T_0$.

\begin{prop}\label{or-I(T)}
Consider the above setting. % and orient $H$ as a boundary component of $W^s(e)$ with multiplicity $\mu$. 
Then $H$ contributes to 
a boundary stratum in the closure of the multi-intersection $I(T)$ with the multiplicity $(-1)^{\ep_j}\mu(z)$ where 
\begin{equation}\label{sign-partial}
\ep_j= n+j-1+\sum_{i=1}^j (\dim W^s(x_i)-n)\,.
\end{equation}
\end{prop}

Note the difference between Definition \ref{multiplicity} and Proposition \ref{or-I(T)}: the breaking of orbits
takes place just below (resp. above) the considered multi-intersection in the first  (resp.  second) case.  In the latter, $H$ contributes to the differential $\p^+$ of the chain that $I(T, x_1, ..., x_d)$ represents
in the complex $C^+_*$.\\

\proof It consists of a generalization to fiber products of the sign given in the case of a product by Proposition
\ref{pro-ori}. A codimension-one stratum remains so through fiber products of transverse mappings and transverse
intersections.

Note that the sign we are interested in is invariant 
by sliding the edges (see the \emph{transition} move on Figure \ref{Figure4}) as  long as 
%they remain to the left of $e$.
the edges to the left of the leftmost complete path (that is, from the root to a leaf) which contains $e$ are not involved.

After a well-chosen sequence of such transitions, $T$ has the following form: $T(v_1)=T_\al\mathop{\vee}\limits_{v_1} T_\beta$.
Here, $v_1$ is the first interior vertex above the root of $T$ and $T(v_1)$ stands for the union of edges above $v_1$;
 the $\vee$ means the bouquet; $T_\al$ is a tree with $j$ leaves
and the edge root of $T_\beta$ is $e$.  
In that case, the multi-intersection becomes a 
usual---not iterated---fiber product and its left factor
has a dimension equal to $\dim I(T_\al, x_1, ...,x_j)+1$. The dimension formula (\ref{dimension}) and Proposition 
\ref{pro-ori} yield the desired sign.\bull

\nd {\bf Orientation and gluing.} Here we go back  
to the setting of subsection \ref{gluing}. We will prove the following statement:
\begin{prop}\label{opposite-or}
The two multi-intersections $I(T')$ and $I(T'')$, equipped with their natural
orientations, give $I(T)$ two opposite boundary orientations. In other words, the amalgamation 
 $I(T')\mathop{\cup}\limits_{I(T)} I(T'')$ is made in the category of oriented manifolds.

\end{prop}

\proof First we observe that the decoration has no effect on  orientation matter. So,
without loss of generality, we may assume $X_{e'}=X_{e''}$ (notation of Figure \ref{Figure4}).
We are now going to use formulas (\ref{I(v,T)}) and (\ref{I(v,T')}) from  the proof of Proposition \ref{prop-gluing}.
Each factors in the iterated fiber product diagram (\ref{I(v,T)}) is contained in some $M^{\times q}$ 
and the maps in the diagram of fiber product 
are induced by the first coordinate in each factor. In coordinates,
a point $a\in I(v,T)$
 reads:
 \begin{equation}\label{I(v,T)-coord}
 \begin{array}{lc}
\quad  &a=\left\{(x_1,\ldots, x_k)(y_1,\ldots, y_\ell)(z_1,\ldots,z_m)\right\}\\
& x_1=y_1=z_1
 \end{array}
 \end{equation}
 where each coordinate $x_i, y_i$ or $z_i$
 denotes a point in $M$.
 Any point $a'$ %(close to $a$ or not) 
 in $I(v,T')$ reads 
 \begin{equation}\label{I(v,T')-coord}
 \begin{array}{lc}
 \quad& a'=\left\{(s',t') (x_1,\ldots, x_k)(y_1,\ldots, y_\ell)(z_1,\ldots,z_m)\right\}\\
& t'=x_1=y_1,\ s'= z_1,
 \end{array}
 \end{equation} 
 where the new coordinates $(s',t')$ are those of $G_{e'}$, source and target.
Similarly, any point $a''$ % close to $a$ 
in $I(v,T'')$ reads:
  \begin{equation}\label{I(v,T'')-coord}
  \begin{array}{lc}
 \quad&a''=\left\{ (x_1,\ldots, x_k)(s'',t'')(y_1,\ldots, y_\ell)(z_1,\ldots,z_m)\right\}\\
 &s''=x_1, \ t''=y_1= z_1,
 \end{array}
 \end{equation}
 where the coordinates $(s'',t'')$ are those of $G_{e'}$, source and target.
 When comparing formulas (\ref{I(v,T')-coord}) and (\ref{I(v,T'')-coord})  at a point of $I(v,T)$, we get $s''=t',\ t''=s'$. 
 %nothing is changed except the permutation of coordinates $(s,t)$.
  This corresponds to reversing the time of the flow 
 of $X_{e'}=X_{e''}$. Then, the time is the only variable whose orientation is changed. 
 This is the reason why the orientation of $I(T)$ changes depending on $I(T)$ is seen as 
 a boundary of $I(T')$ or $I(T'')$.
 %for the change of orientation of  $I(T)$ as boundary of $I(T')$ or $I(T'')$.
 The change of the place of the couple (source, target) has no effect on the orientation 
 since it is an equidimensional couple.
 \bull
 
 \section{$\A$-structure}\label{A-infty}
 In this section, we exhibit how one can construct an $\A$-structure on the Morse complex $A=C_*(f,X^+)$  whose 
 first operation $m_1$ coincide with the differential $\partial^+$. 
The grading is now defined  by setting $|x|:=n-\dim W^s(x)$ for every critical point $x\in crit f\cup crit^+f_\p$. 
 Note this grading is {\it cohomological}, that is, the degree of the differential $m_1=\partial^+$ is $+1$.
 
 One fixes a  \emph{coherent} sequence $\g_\infty$ in the group $G$. Every Fukaya tree $T$ is endowed with 
 the $\g_\infty$-standard decoration $\mathcal D(T)$.
   So, for  any tree $T$ with $d\geq 2$ leaves and any sequence of $(d+1)$ critical points  $x_1,\ldots x_d,x_{d+1}$ 
   (with possible repetition) we have the multi-intersection  $ I(T ,x_1,\ldots x_d; x_{d+1})$  defined by  the following fiber product (compare to the $T$-evaluation map  defined at the end of subsection \ref{evaluation}):
 \begin{equation} \label{Tevaluation}
 	I(T, x_1,\dots,x_d; x_{d+1}):= %<I(T, x_1, ..., x_d), x_{d+1}>= 
	\lim\left(I(v_{root}^1)\mathop{\longrightarrow}\limits^{^{p_{root}}} M\mathop{\longleftarrow}\limits^{j}
	W^u(x_{d+1},X^+)\right)\
 \end{equation} 
 We recall that $v^1_{root}$ is the terminal vertex of the edge originating at the root and the associated generalized 
 intersection is defined inductively in Section \ref{towards}.
 Since $\mathcal D(T)$ is admissible, this set is a manifold.
 Using the dimension formula of Proposition \ref{dim-formula}, we conclude that  its dimension is 
\begin{equation}\label{equ-dimen}
\dim I(T,x_1,\dots, x_d; x_{d+1})=  
d-2+\vert x_{d+1}\vert-\sum_{i=1}^k \vert x_i\vert .
\end{equation}

Therefore the dimension of $I(T,x_1,\cdots x_d;x_{d+1})$ is zero if and only if 
$$
\vert x_{d+1}\vert =\sum_{i=1}^k \vert x_i\vert + 2-d
$$
In what follows, we denote by $\mathcal T^0$ (resp. $\mathcal T^0_d$) the set of generic Fukaya trees (resp. with $d$
leaves).
For $d\geq 2$,  we define the linear maps $m_d:A^{\otimes d}\to A$ by  
\begin{equation}\left\{
\begin{split}
&m_d(x_1,\dots, x_d):=\\
&(-1)^{\sum_{s=1}^d  (d-s)\vert x_s\vert } %d\vert x_d\vert +(d-1) \vert x_{d-1}\vert+ \cdots+\vert x_{1}\vert }
 \sum _{T \in \mathcal T^0_d}\left[\sum_{%z\in A; 
|y|=2-d+\sum |x_i|} \#I(T,x_1,\dots,x_d; y) y\right].
\end{split}\right.
\end{equation}
One should think of $I(T,x_1,\dots,x_d; y)$ as an oriented zero-dimensional manifold and\hfill\break
$\# I(T,x_1,\dots,x_d; y)$ is the algebraic number of signed points in this manifold.
%obtained by the adding up the orientation, i.e. $\pm1$, at the points.  
 It is clear from the definition that the degree of $m_d$ is  $2-d$.

\begin{rien}{\bf Geometric definition of the first operation.} %$m_1$.}
 {\rm So far, we have not considered trees with just one leaf. Nevertheless,
for $x\in crit f\cup crit^+ f_\partial$,
one can define $m_1(x)$ geometrically in the following way.
 From the compactification of $W^s(x,X^+) $ one extracts the 
\emph{frontier} $F^s(x)$ which is the complement of  $W^s(x,X^+) $  in its closure.\footnote{ Here, $X^+$ could be replaced with any $C^\infty$-approximation.} As $X^+$ is Morse-Smale,
$F^s(x)$ is transverse to $W^u(y, X^+)$ for every $y\in crit f\cup crit^+ f_\partial$. 
When $|y|=|x|+1$, one defines the 0-dimensional intersection manifold 
$I(x; y):= F^s(x)\cap W^u(y, X^+)$. As it is oriented, it is made of a finite set of signed points. Then one defines
\begin{equation}
m_1(x)= \sum_{|y|=|x|+1} \# I(x; y)y\,.
\end{equation}  
}
\end{rien}

\begin{thm}\label{thm-Ainfty}
$ (A,m_1,m_2,\dots )$ is an $\A$-algebra.
\end{thm}

\begin{proof}
The $\A$-relations read for every $d>0$:
\begin{equation} \label{def-A1}
\sum_{j,k, l} (-1)^{j +k\ell} m_{j+1+\ell}(1^{\otimes j}\otimes m_k \otimes 1^{\otimes l})=0
\end{equation}
where the sum is taken over all non-negative integers $j,k,l$ such that $j+k+l= d$.

When putting entries $(x_1,\dots, x_d)$, new signs appear according to Koszul's rule:
\begin{equation}
\begin{split}
&(1^{\otimes j}\otimes m_k)(x_1,\ldots, x_j,x_{j+1},\dots, x_{j+k})=\\
 &(-1)^{(|x_1|+\dots+|x_j|)|m_k|}
\bigl(x_1,\ldots, x_j, m_k(x_{j+1},\dots, x_{j+k})\bigr),
\end{split}
\end{equation} 
and Identity (\ref{def-A1}) becomes:
\begin{equation} \label{def-A2}
\sum_{j+k+l=d} (-1)^\ep 
m_{j+1+l}(x_1,\dots,x_j, m_k(x_{j+1},\dots, x_{j+k}),x_{j+k+1},\dots, x_d)=0,
\end{equation}
where $\ep=j +kl+(2-k)(\sum_{i=1}^j|x_i|) $.

By the very definition of the $m_k$'s, % and Remark \ref{rem-coherence},
the above $\A$-relations  are  equivalent  to 
the following  identities
\begin{equation} \label{eq-A-proof1}
 \sum_{j, k,T_{\alpha}, T_{\beta}}\sum_{ y} (-1)^{\ep'} %jk+ (d-j-k)} 
 \# I(T_{\alpha},x_{j+1},\dots,x_{j+k};y)\, \#I(T_{\beta},x_1,\dots, x_j,y,x_{j+k+1},\dots x_{d}; z)=0
\end{equation}
for all $d \geq 1$ and all sequence $(x_1,\dots, x_{d}, y, z)$ of critical points  with $|y|=2-k+\sum_{i=j+1}^{j+k} |x_i|$, 
$|z|=3-d +\sum_{i=1}^{d}|x_i|$ and $\ep'=\vert x_1\vert +\cdots \vert x_j \vert -j$.
In this sum, $T_\alpha$ is a generic tree  with $k$ leaves and $T_\beta$ is
a generic tree with $d-k+1$ leaves. By (\ref{equ-dimen}), the manifold 
$I(T_{\beta}, x_1,\dots, x_j,y,x_{j+k+1},\dots x_{d}; z)$ is 0-dimensional.

Note that, from  
$|z|=3-d+\sum_{i=1}^{d}|x_i|$, it follows that for every generic tree $T$ the
multi-intersection   $I(T, x_1,\dots,x_{j+k+l}; z)$ is  one-dimensional. 

The proof of 
(\ref{eq-A-proof1}) will follow from  analysing the frontier of this \emph{oriented} manifold in its compactification
(we know from Appendix \ref{appendix-comp} that the multi-intersections are compact manifolds with 
$C^1$ conic singularities.)

We fix a generic tree $T$ with $d$ leaves and consider the compact 
1-dimensional submanifold with conic singularities
 $ cl(I(T,x_1,\dots, x_d;z))\subset M^{\times (d-1)}$.  
  By blowing up the singular points, such a manifold can be 
 thought of as a manifold with boundary where some boundary points are identified. 
 Such a point $P$ is  equipped with a sign
 which is the sum of the  boundary-orientation signs of the inverse images of $P$ in the above blowing up, which is itself oriented.  
 Therefore, we have:

\begin{equation} \label{equ-sign}
	\sum_{P\,\in\, \partial \bigl(cl(I(T,x_1,\dots, x_d;z))\bigr) }sign(P)=0
\end{equation}

By iterating Proposition \ref{comp-graph}, % (see also Appendix \ref{appendix-comp})}, 
the boundary components of the closure $cl(I(T,x_1,\dots, x_d;z)\bigr)$
 are divided into three types:\\

\noindent {\sc Type A:}  The boundary components coming from the broken orbits in the compactification of
the generalized stable manifold $W^s(e, X_e)$. Here, $e$ is  
an interior edge in the tree $T$ and $X_e$ is its decoration.  Such
a codimension-one stratum involves some critical point $y$ and its %stable/unstable
 invariant manifolds with respect to $X_e$. 
 Therefore, it is of the form %
$$I(T_{\alpha},x_{j+1},\dots,x_{j+k};y)   \times I(T_{\beta},x_1,\dots, x_j,y,x_{j+k+1},\dots x_{d}; z),\  
0\leq k\leq d.$$ The first (resp. second) factor in this product  comes from the unstable (resp. stable) manifold of $y$.
The tree $T$ is equal to the {\it connected sum}
 $$T=T_{\alpha} \#_{j+1}T_{\beta}$$ 
where the root of $T_{\alpha}$ %,  $T_{\beta}$ 
is glued to the  $(j+1)$-th leaf of $T_\beta$.\\

 \noindent {\sc Type B:} The boundary components of the form $I(T/e,x_1,\dots ,x_d;z)$
 where $e$ is an interior edge of $T$ and $T/e$ denotes the tree obtained from $T$ by collapsing $e$ to a point. 
 They are induced by the diagonal of $M\times M$ except over the zeroes of the vector field  $X_e$  
 which decorates $e$.\\
 
 \noindent{\sc Type C:} The boundary components which are induced by $\p M$. In general, a stable 
 manifolds has orbits coming from $\partial M$.  
  Actually, the type-C components are empty in the considered multi-intersection $I(T,x_1, \dots, x_d; z)$. Indeed,
  by construction, the unstable manifold $W^u(z, X^+)$ lies in the interior of $M$ except very near $z$
  if $z\in crit ^+f_\p$. Thus, 
  the multi-intersection $I(T,x_1, \dots, x_d; z)$, that is
   the {\it evaluation} $< I(T,x_1, \dots, x_d), z>$, %(compare Subsection \ref{rien-evaluation}), 
   has no type-C boundary components.\\

 Therefore the identity (\ref{equ-sign}) splits into the sum of two terms % two sums
  \begin{equation} %\label{eq-A-proof1}
  S_A+S_B=0 
\end{equation}
where $S_A$ (resp. $S_B$) is the contribution of the type-A (resp. type-B) components. Note that $S_A$ is exactly
the left handside of Equation (\ref{eq-A-proof1}) since $T$, $j$ and $k$ determine $T_\al$ and $T_\beta$.
 Therefore, we are reduced to prove the nullity of $S_B$.

By Proposition \ref{prop-gluing}, a type $B$ boundary component $I({T/e},x_1,\cdots,x_d; z)$  appears as a boundary component of exactly one another one-dimensional intersection submanifold  $I({T'},x_1,\cdots,x_d; z)$  where $T'$ is the unique generic tree, distinct from $T$, obtained from $T/e$  by an expansion at its unique degree-$4$ vertex 
(see Figure \ref{Figure4}). 
  Moreover, by Proposition %\ref{prop-gluing} and \ref{pro-ori}
   \ref{opposite-or}, the induced orientations are opposite.  Therefore,
  in the sum $S_B$ %*0=\sum_{T \text{ generic tree}} 0= \sum_{T \in \mathcal T^0}\sum_{P\in \partial \overline{I_T(x_1,\dots, x_d, z)} }sign(P)$
   these two  terms cancel each other out. \\
   
   \nd{\sc Checking of the signs.} We apply Proposition \ref{or-I(T)} which gives us 
   the following sum of chains of geometric nature (without evaluating):
   \begin{equation}
   \partial^+ I(T,x_1, \dots, x_d)=
    \sum_{j, k, y} (-1)^{\ep_j} \#I(T_\al,x_{j+1},\dots, x_{j+k}; y)\,  \partial ^+I(T_\beta,x_1, \dots,y,x_{j+k+1},\dots, x_d).
    \end{equation} 
   Here, $j$ varies from 1 to $d-1$; $k$ from 1 to $d-j$; $y$ is a critical point such that 
    $|y|= 2-k+\sum |x_i|$ and the geometric sign is the one given by formula \ref{sign-partial}, that is,
  $$\ep_j= n+j-1+ \sum_{i=1}^j (\dim W^s(x_i)-n)=n+j-1-\sum_{i=1}^j\vert x_j\vert.$$

\bull
\end{proof}

\section{Morse concordance and homotopy of $A_d$\,-structures}\label{homotopic-structures}

We have seen that the operations $m_1, m_2, \dots$ which define an $A_\infty$-structure on
the complex  $A:=C_*(f,X^+)$
are determined by the choice of a family of coherent 
 decorations for every Fukaya tree $T$. %with  $d$ leaves at most; 
%The first  $m_1$ is not listed since it is a part of the definition of the complex under consideration. 
Recall that a decoration of an edge $e$
is a vector field $X_e$ approximating $X^+$. In particular, it lies in the same connected component 
of Morse-Smale vector fields.

Assume we have two coherent sequences $\g_\infty$ and $\g'_\infty$ in the group $G= Di\!f\!f_0(M)$
and their associated standard
decorations %and transition-compatible families 
$\{\mathcal D(T)\}_T$ and $\{\mathcal D'(T)\}_T$, decorating all 
Fukaya trees. In general, these two families give  rise to two distinct $\A$-structures $(m_1,m_2,\dots)$ and
$(m'_1,m'_2,\dots)$. We are going to show that these two structures can be linked by a {\it homotopy} thanks
to multi-intersections over the product manifold $\hat M:= M\times [0,1]$ which is a manifold with boundary and corners.
Note that the complex $C_*(f,X^+)$
 is kept unchanged; in particular, $m_1=m'_1$.
A multi-intersection $\hat I(T)$ over $\hat M$ associated with a decoration 
$\hat{\mathcal D}(T)$ will be thought of as a cobordism from its trace over $M\times\{0\}$ to its trace over $M\times\{1\}$.
Such a family of  cobordisms will be called a {\it geometric homotopy}. The expression {\it Morse concordance}
from the section title
emphasizes the fact that the underlying manifold is a product $M\times[0,1]$ equipped with a function without critical 
points in its interior.

Here, we are inspired by Conley's continuation map \cite{conley} that we have extended to the $\A$-case. 
The case of the 
Morse complex is discussed in \cite{Schw} and  \cite{Hutch1} as a prelude to the (infinite dimensional) case  of Floer 
homology. In fact, Andreas Floer \cite{floer} had first evoked the idea for the infinite dimensional Morse Theory.  
The invariance of the Morse homology was proved earlier using other methods.

 \begin{rien} {\bf Construction of a Morse concordance.}\label{cobordism}
 
{\rm For simplicity, we first restrict to the case where $\partial M$ is empty. Then, $\hat M:= M\times [0,1]$
 is a manifold with boundary. The general
 case will be sketched in Remark \ref{boundary-corners}. 
 When $\partial M=\emptyset$, we consider a Morse-Smale %(of type +) 
 pseudo-gradient $X$
 adapted to the Morse function $f$. % We are given two coherent decorations $\mathcal D$ and $\mathcal D'$ of $T$. From this data, 
 
 We first build a Morse function $\hat f$ on $\hat M$ with no critical points in the interior of $\hat M$ 
 whose restriction to $M_i:= M\times\{i\}$, $i= 0, 1$, reads $\hat f\vert_{M_i}= f+c_i$ where $c_i$ is some constant.
  More precisely, one requires the critical points of $\hat f\vert_{M_0}$ to be of type $+$ %Dirichlet type 
  and those of 
  $\hat f\vert_{M_1}$ to be of type $-$. %Neumann type. 
 The pseudo-gradient vector 
 fields adapted to $\hat f$  are required to be tangent to the boundary. 
  This needs a slight modification with respect to the Morse theory we have considered so far. 
  %This modification is also emphasized when speaking of {\it Morse concordance}; that does not exactly dealwith a Morse theory on a manifold with boundary.
  }
 \end{rien}
 Let $h: \R\to \R$ be the Morse function defined by $h(t)=(2t-1)^3 -3(2t-1)$; its critical points are $t=0, 1$.
 For $(x,t)\in \hat M$, set $\hat f(x,t)= f(x)+ h(t)$. 
 %If $(X_t)_{t\in [0,1]}$ is a path of pseudo-gradients adapted to $f$,define a pseudo-gradient $\hat X$ of $\hat f$ by the formula $\hat X(x,t)= X_t(x)+\nabla h(t)$. 
 If $a$ is a critical point in $M_0$ (resp. in $M_1$), we have:
 \begin{equation} 
 Ind(\hat f, a)= Ind(f,a)+1\quad(\text{resp. } =Ind(f,a))\,.
 \end{equation} If $\hat X$ is a pseudo-gradient on $\hat M$ adapted to $\hat f$ and tangent to $M_0\cup M_1$,
 depending on $a\in M_0$
 (resp.  $a\in M_1$),  
 the stable manifold $W^s(a, \hat X)$ (resp. the unstable manifold $W^u(a, \hat X)$)
 meets the interior of $\hat M$; on the contrary, the unstable (resp. stable) manifold
  lies entirely in $M_0$ (resp. in $M_1$).

 The critical points of $\hat f$ in $M_0$ will serve as {\it entries}; those lying in $M_1$ will be used as {\it test data}.  
 Consider now an edge $e\subset T$ and its two decorations $X_e\in \mathcal D(T)$ and $X'_e\in \mathcal D'(T)$.
 Since $X_e$ and $X'_e$ are approximations of the same Morse-Smale vector field $X^+$ on $M$
  it is possible to join them by a path $\left(X^t_e\right)_{t\in [0,1]}$ of Morse-Smale vector fields and 
  form the vector field $\hat X_e$ on $\hat M$ defined by 
 \begin{equation}
 \hat X_e(x,t):= X^t_e(x)+\nabla h(t).
 \end{equation}
 This is a baby case of a method initiated by A. Floer.\footnote{ Floer \cite{floer} has introduced this method for finding the 
 so-called {\it continuation morphism} which connects two (Floer) complexes built from different  data.} 
 Assume moreover that the path $\left(X^t_e\right)_{t}$ is stationary for $t$ close to 0 and 1 in order that $X_e$
is adapted to $\hat f$ near each critical point.
 Generically on the collection of  paths $\left(X_e^t\right)_{t\in [0,1]}, e\subset T$, some transversality conditions
 may be fulfilled which allow us to construct recursively the following (see Section \ref{towards}):
 \begin{itemize}
 \item the generalized stable manifolds $\hat W^s(e)$ associated with the edges $e$ of $T$,

\item the  multi-intersections  $\hat I(v)$ of stable manifolds associated with the vertices $v$ of $T$,
 \end{itemize} 
both of them being transversely defined. 
In other words, the decoration $\hat{\mathcal D}(T)$ made of the collection $\{\hat X_e\}_{e\subset T }$ is 
chosen admissible (Definition \ref{requirements}). 

 Then,
for every vertex $v$ in $T$ %the considered Fukaya tree, 
the manifolds $\hat I(v, \hat{\mathcal D}(T))$ are transverse to  $p_1^{-1}(M_i), i=0,1$, where $p_1$ denotes the
first projection $\hat M^{\times n(v)}\to \hat M$. Thus, we have proved the following:

\begin{prop} For every vertex of $T$,
the multi-intersections $\hat I(v, \hat{\mathcal D}(T))$ is a cobordism from 
$I(v, \mathcal D(T))$ to $I(v,\mathcal  D'(T))$. This cobordism extends to a stratified cobordism between 
their respective compactifications.
\end{prop}

For the definition of the $\A$-operations, it is crucial that the family of chosen decorations is  
coherent in the sense of Section \ref{coherence}.

\begin{prop} The set of decorations $\{\hat{\mathcal D}(T)\}_T$, where $T$ ranges over the Fukaya trees %having $d$ leaves at most, 
can be chosen in order to be coherent over $\hat M$.
\end{prop}

\proof   The problem of coherence can be solved by using the same method as in Section \ref{coherence}
and performing it ``over $M\times[0,1]$'', that is, replacing the group $G$
of diffeomorphisms of $M$, isotopic to $Id_M$, by the group $\hat G$ 
of diffeomorphisms of $\hat M$ isotopic to $Id_{\hat M}$;  the elements  $\hat g\in \hat G$ are not required 
to preserve the level sets $\{t=cst\}$. 

There are given two coherent sequences in $G$, namely $\g$ and $\g'$ respectively attached to $M_0$ and $M_1$.
The issue is to extend the pair $(\g,\g')$ to $\hat M$ so that the extension $\hat\g$ is coherent 
in the sense of Definition \ref{coherence-defn}. This can be performed by the \emph{translation flow} method 
introduced in Proposition \ref{translation-prop} and applied in Propositon \ref{prop-coherence}. 
This method admits a relative version since it proceeds 
in  successive extensions of disc bundle sections over increasing dimension skeleta.
\bull

\begin{remarque}\label{boundary-corners}
{\rm When $M$ has a non-empty boundary and we look (for instance) at the critical points of type $+$,
$\hat M$ has  corners modelled on $\R^{n-1}\times Q$ where $Q$ is a quadrant in the plane and there are critical 
points of $\hat f$ lying in the corners. The only issue is to define what is an adapted pseudo-gradient, in order that 
the stable manifolds are well defined. One solution consists of demanding the pseudo-gradient to point inwards
along $\partial M\times[0,1]$, except near the critical points in the corners where it is tangent to
 $\partial M\times[0,1]$. The rest of the previous discussion is similar.
}\end{remarque}

We are going to see that the above geometric cobordisms lead to a quasi-morphism of the $A_\infty$-structure
defined thanks to the set of decorations $\{D(T)\}_T$ to the one defined by $\{D'(T)\}_T$. The required uniqueness
up to homotopy
will follow.

\begin{rien} {\bf Construction of  $A_\infty$-quasi-isomorphism.}

{\rm We now construct a quasi-isomorphism $\{\vp_d\}_{d\geq 1}$  between the $\A$-structures $(m_d)_{d\geq 1}$ and 
$(m'_d)_{d\geq 1}$ on $A= C_*(f,X^+)$ corresponding to the decorations  ${\mathcal D}(T)$ and ${\mathcal D}'(T)$ as 
they were introduced in Section \ref{A-infty}.
In fact the construction of  $\vp_d:A^{\otimes d}\to A$ is very similar to that of the $m_i$'s. 
}
\end{rien}

For $d+1$ critical points $x_1,\dots x_d, y$ of $f$,
we define the multi-intersection submanifold of $\hat{M}^{\times (d-1)}$ 
\begin{equation} \label{Tevaluation}
 	\hat{I}(T,x_1,\dots,x_d; y):= 
	\lim\left(I(v_{root}^1)\mathop{\longrightarrow}\limits^{^{p_{root}}} \hat{M}\mathop{\longleftarrow}\limits^{j}
	W^u\bigl((y,1) ,\hat{X}^+)\bigr)\right)
 \end{equation}
which is defined using the decoration $\{\hat{{\mathcal D}}(T)\}_{T\in {\mathcal T}_0}$. Here, the inputs of 
 $\hat I(T)$ are the $ (x_i,0)$'s and the output is $(y,1)$.

For $d\geq 1$,  we define $\vp_d:A^{\otimes}\to A$ by 
\begin{equation}\left\{
\begin{split}
&\vp_d(x_1,\dots, x_d):= \\
&(-1)^{ \sum_{s=1}^d (d-s)\vert x_s\vert  }\sum _{T \in \mathcal T^0_d}\left[\sum_{ 
\Vert(y,1)\Vert=2-d+\sum \Vert (x_i, 0)\Vert} \#\hat{I}(T,x_1,\dots,x_d; y) y\right].
\end{split}\right.
\end{equation}
where degree $\Vert .\Vert$ is defined with respect to $\hat{f}$ as a Morse fonction on $\hat{M}$.

Note that condition $\Vert(y,1)\Vert=2-d +\sum \Vert(x_i, 0)\Vert $ 
is the necessary and sufficient condition for zero dimensionality of $\hat{I}(T,x_1,\dots,x_d; y)$.  
Moreover, by observing that 
\begin{equation}
\Vert(y,1)\Vert=n+1-Ind(\hat f,(y,1))=n+1-Ind (f,y)=|y|+1
\end{equation}
and
\begin{equation}
\Vert(x_i,0)\Vert=n+1-Ind(\hat f,(x_i,0))=n+1 -(Ind (f,x_i)+1)=	|x_i|
\end{equation}
we conclude that the degree of $\vp_d$ is $1-d$, ({\it i.e.} one lower than  $m_d$ and $m'_d$). Let us also recall that
when $\partial M\neq\emptyset$ and $x\in crit^+f_\partial$,  we have $\dim W^s(x, X^+)= Ind(f_\partial,x)+1$.

\begin{prop} The collection $(\vp_1, \dots, \vp_d,\dots)$ defines a quasi-isomorphism of $\A$-structures.
\end{prop}

\begin{proof} It is easily checked that $\vp_1: A\to A$ is the identity. Then, as soon as the morphism relations are fulfilled,
we get a quasi-isomorphism. Let us recall these relations from Appendix \ref{appendix-Astr}:
\begin{equation}\label{condition—morphism1}
 \sum_{j+k+l=d}  (-1)^{j+kl} \vp_{j+l+1}\ (1^{\otimes j} \otimes m_k\otimes  1^{\otimes l})=\sum_{k=1}^d \sum_{i_1+\cdots +i_k=d} (-1)^{\epsilon_{i_1, \cdots i_k}}m'_k( \vp_{i_1} \otimes \cdots \otimes   \vp_{i_k})
 \end{equation}
 \nd where $\epsilon_{i_1, \cdots i_k} = \sum_{j=1}^k (k-j)(r_j-1) $.
 These relations are 
 implied by  geometric information  given by the decoration family $\{\hat{\mathcal D}(T)\}_T $, namely,
 for every $d>0$,
 \begin{equation}\label{geom}
 \left\{
 \begin{split}
 &\sum_{j,T_\al,T_\beta} \sum_{(y,0)}(-1)^{j-\sum_{i=1}^j\vert x_j\vert} \#I^0(T_\al, x_{j+1},\dots,x_{j+k}; y)
 \#\hat I(T_\beta, x_1,\dots,x_j,y, x_{j+k+1},\dots,x_d;z)\\
&= \sum_{k=1}^d \sum_{
 \tiny\begin{array}{c}
 T_\ga\in \mathcal T^0_k\\
  0<i_1<\cdots <i_k<d       % i_1+\cdots+i_k= d
 \end{array}
} 
\left( \#I^1(T_\ga,y_1, \dots, y_k; z)
 \prod_{ \tiny\begin{array}{c}
 j=1\\
T_j\in \mathcal{T}^0_{i_j-i_{j-1}}
 \end{array}}^{k+1}\#\hat I(T_{j},x_{i_{j-1}+1},\dots, x_{i_j}; y_j) \right).
 \end{split}
 \right.
 \end{equation}
 where   $i_0=0$, $i_{k+1}=d$ .  
 Here, $I^0(T,-)$ (resp. $I^1(T,-)$) stands for the multi-intersection computed with the decorations $\{\mathcal D(T)\}_T$ 
 on $M_0$ (resp. $\{\mathcal D'(T)\}_T$ on $M_1$);  the connected-sum  tree  $T= T_\al\#_{j+1} T_\beta$ is a generic tree with $d$ leaves.  Note that the $i_k$'s in the identity (\ref{condition—morphism1})  correspond  to the quantities $i_j-i_{j-1}$ in (\ref{geom}).

 The proof of (\ref{geom}) is similar to the proof of (\ref{eq-A-proof1}) with some new phenomena. By degree arguments, one knows that the multi-intersection $\hat I(T,x_1,\dots, x_d; z)$ is one-dimensional.
 So, we have to analyze its compactification. We already know that the collapse of an edge of $T$ contributes to zero
 because such a boundary component appears twice in the considered sum
  with opposite orientations. The boundary component 
 $\partial M\times[0,1]$ contributes also to zero as the vector field $\hat X$ points inwards except in 
 a very small neighbourhood of $crit^+f\times\{0,1\}$.
 
 The first new phenomenon is the following. The breaking of an orbit of $\hat X_e$ involves
 in the same time the boundary of $\hat M$:
 if it breaks in $y\in M_0$, the unstable manifold $W^u(y, \hat X_e)$ coincide with $W^u(y, X_e)$. This explains the 
 factor $\#I^0(T_\al, x_{j+1},\dots,x_{j+k}; y)$ in the left handside of (\ref{geom}). 
 % Moreover, a double breaking involving two critical points one of both 
 %lying in $M_0$ is not generic in the frontier of $\hat  I(T,x_1, \dots, x_d)$ because $W^s(y,\hat X_e)\times W^u(y, X_e)$ generates a smooth boundary component of $\hat I(T,x_1,\dots, x_d)$.
 
 The second new phenomenon is that, if the breaking happens
  at $y\in M_1$ and $d>1$, then the breaking cannot happen alone. 
 Indeed, $W^s(y, \hat X_e)$ is contained in $M_1$; therefore, it has an empty intersection with any other stable manifold
 (or generalized stable manifold) which, by construction, lies in $int(\hat M)\cup M_0$. 
 Assume the root of $e$ is not the root of $T$ and  let $e'$ be the other edge of $T$
 having the same root as $e$. Then, we have proved that the \emph{generalized} stable manifold $W^s(e', \hat X_{e'})$
 must also be contained  $M_1$ or (over $M_1$ through $p_1$ in the fiber product construction.) By iterating this argument, one proves the following claim.\\
 
 \nd{\sc Claim.} {\it If $d>1$, any non-empty 
connected component $C$ of the frontier\footnote{ The frontier of a multi-intersection consists of its
 compactification with the multi-intersection in question removed.} of  $\hat I(T,x_1,\dots, x_d)$ which involves
 the breaking of an orbit at a zero in $M_1$
and no breaking in $M_0$
 gives rise to the following decomposition of $T$: there exist $k>0$, some edges $e_1, \dots, e_k$ in $T$ separating the 
 root of $T$ from all leaves and points $y_1,\dots, y_k$ in $M_1$  which are respectively zeroes of $X'_{e_j}$, 
 $j= 1,\dots, k$, 
 such that $C$ is contained in the multi-intersection $I(T^1,y_1, \dots, y_k)$.}\\
 
 In particular, except when some $y_j$ is of maximal Morse index (which has a neutral effect), $C$ is of  codimension $k$
 in the compactification of $\hat I(T,x_1,\dots, x_d)$. If $k>1$, such a $C$ does not adhere to any smooth boundary 
 component. This phenomenon is compatible with the fact that the singularities of the compactification are \emph{conic}.
 This claim gives the geometric signification of the right handside of (\ref{geom}) and finishes the proof {\bf up to sign}.
 It explains that a one-dimensional intersection of  $\hat I_T(x_1,\dots, x_d)$ 
 with $W^u(z, X^+)$ cannot generically avoid to have singular points in
 such a stratum.  All other configurations of orbit breaking are generically avoidable, and hence, do not appear 
 in the counting of  (\ref{geom}). \bull
\end{proof}

\appendix \label{appendix}
\section{Complements on  the submanifolds with $C^1$ conic singularities}\label{appendix-comp}

\begin{defn} \label{skel} Let $K$ be a compact submanifold of $M$ with $C^1$ conic singularities. The $k$-\emph{skeleton}
 $K^{[k]}$ of $K$ is the union of the strata of $K$ which are of codimension at least $n-k$ in $M$.
 \end{defn}

  \begin{lemme} \label{lemma-conic} Let $A$ and $B$ two compact submanifolds of $M$ with $C^1$ conic 
  singularities.
 Then, for a generic diffeomorphism $g$ of $M$
 %there exists a residual set $R_{A,B}\subset Di\!f\!f(M)$, actually dense in the $C^\infty$ topology and open in the $C^1$ topology, such that  for every $g\in R_{A,B}$  
 the image $g(A)$ is transverse to $B$, meaning that 
 each stratum of $g(A)$ is transverse to every stratum of $B$.  
 Moreover, this transversality is fulfilled in an open set of the $C^1$ topology of $Di\!f\!f(M)$.
  
 \end{lemme}
 
 \proof Thanks to the group action, it is enough to prove this statement near the Identity of $M$.
 Assume the skeleton $A^{[k]}$ is already transverse to $B$. So, near any point $x\in A^{[k]}$, each stratum of 
 $A$ is transverse to $B$. This fact %, already mentioned, 
 directly follows  from the conic transverse structure.
 
Let $S$ be an open $(k+1)$-stratum of $A$; it is transverse to $B$ outside some compact set $C\subset S$. By 
the  very first transversality theorem of Thom \cite{thom-cob}, an arbitrarily  small 
ambient isotopy supported in a neighborhood of $C$ makes $S$ successively transverse to 
the 0-skeleton, the 1-skeleton, and so on, until being transverse to $B$. 
Moreover, these smooth approximations of $Id_M$ fulfilling the above requirement form a $C^1$ open set.
This double induction gives the desired genericity, including the $C^1$ openness of transversality.\bull

\begin{lemme}\label{union} Let $A$ and $B$ be two submanifolds with conic singularities which are
mutually transverse. Then their 
union $A\cup B$ is a submanifold with $C^1$ conic singularities.  Its strata are of one of the following forms
%for each pair $(S,\Si)$
 where $S$ is a stratum of $A$ and $\Si$ is a stratum of $B$:
%\begin{itemize} \item[(i)]
 $S\cap \Si$ or 
%\item[(ii)] 
$S\smallsetminus \Si$ or $\Si\smallsetminus S$.
%\end{itemize}
\end{lemme} 
\proof The only matter is about the structure at points in $A\cap B$. 
Set $\La=S\cap\Si$.
For $x\in \La$, let $\Theta_{B, \Si, x}$ 
be the transverse conique structure to $\Si$ in $M$ induced by $B$ on the normal fiber  $\nu_x(\Si,M)$. 
By transversality, we have the equality 
$\nu_x(\Si,M)=\nu_x(\La, S)$. So, $\nu_x(\La, S)$ is equipped with $\Theta_{B,\Si, x}$.
Similarly we have the transverse conic structure 
$\Theta_{A,S,x}$ induced on the normal fiber $\nu_x(\La, \Si)$. % of the normal bundle to $\La$ in $\Si$.

These two transverse structures can be trivialized over a small open neighborhood of $x$ in $\La$.
Since the two bundles $\nu(\La, S)$ and $\nu(\La, \Si$) are complementary in $\nu(\La,M)$,
 the two trivializations are independent.
Therefore, they can be realized by a same $C^1$ diffeomorphism of $M$. Hence, the conic structure 
induced by $A\cup B$ on the normal fiber $\nu_x(\La, M)$ is the \emph{join} 
(in the sense of the piecewise linear topology) $\Theta_{A,S,x}*\Theta_{B,\Si, x}$.
\bull

The next lemma and its corollary can be proved in the same way.

\begin{lemme}\label{product} Let, for $i=1,2$, $A_i\subset M^{\times k_i}$ be a compact submanifold 
with $C^1$ conic singularities. Then $A_1\times A_2\subset M^{\times (k_1+k_2)}$ is so.
\end{lemme}

\begin{cor}\label{fiber-product} In the same product setting as in Lemma \ref{product}, let $p_i: M^{\times k_i}\to M$
be a projection to one factor of the product. If the restrictions $p_1\vert A_1$ and $p_2\vert A_2$ are transverse,
then the fiber product over $M$ of these two maps is a compact submanifold $M^{\times (k_1+k_2-1)}$
with $C^1$ conic singularities.
\end{cor}

  \section{Some applications of Sard's theorem to 
 immediate transversality} \label{sard}
%\nd {\sc Case of  translation flows.}
  Let $S_1$ and $S_2$ be two smooth\footnote{ Smooth stands for $C^\infty$.} submanifolds of positive codimension in $\R^n$, possibly equal.
 One defines the space of \emph{secants} from $S_1$ to $S_2$ by
 %as the set of pairs  $(u,x)\in \vec\R^n\times \R^n$ such that $x\in S_1$ and $x+u\in S_2$. 
 \begin{equation}
Sec_{S_1,S_2} := \{(u,x)\in \vec\R^n\times\R^n\mid x\in S_1 
\text{ and }x+u\in S_2\}
\end{equation}
Let $\pi : Sec_{S_1,S_2}\to \vec \R^n$ denote the projection $(u,x)\mapsto u$. 
 
If $f_1(x)= 0\ (\text{resp. } f_2(x)= 0)$ are two local systems of regular equations 
defining $S_1$ (resp. $S_2$), the---potential---tangent 
space to $Sec_{S_1,S_2}$ at $(u,x)$ is defined by the linearized system
\begin{equation}\label{lin-sec}
\left\{
\begin{array}{l}
Df_1(x)\cdot \de x=0\\ 
Df_2(x+u)\cdot(\de x+\de u)=0.
\end{array}
\right.
\end{equation} 
This system is of maximal rank and hence $Sec_{S_1,S_2}$ is a smooth submanifold.

\begin{prop}\label{b1}
 For almost every $u\in \vec\R^n\smallsetminus \{0\}$\,\footnote{ In this section, ``almost every'' is meant in the Baire sense, that is, ``in some residual subset.''} and every $x\in \pi^{-1}(u)$, 
 the two tangent vector spaces $T_xS_1$ and $T_{x+u}S_2$ are not \emph{coplanar} in the sense that
 there is no 
\emph{hyperplane} containing both of them.
\end{prop}

Note that when ${\rm codim}\,S_1+{\rm codim}\,S_2>n$ and $\pi^{-1}(u)\neq\emptyset$ coplanarity is automatic.\\

\proof Thanks to the smoothness assumption Sard's Theorem is available. It
tells us that almost every $u$ is a regular value of $\pi$ (possibly with an empty inverse image).
But an easy argument shows that $u$ is a critical value of $\pi$ if and only if there exists $x\in \pi^{-1}(u)$
such that the tangent spaces $T_xS_1$ and $T_{x+u}S_2$ are coplanar.

Indeed, in case of coplanarity, there is a non-zero 
linear form $L$ vanishing on $\ker Df_1(x)$ and \break$\ker Df_2(x+u)$.
For $(\de u, \de x)$ solution of (\ref{lin-sec}), we have $L(\de x)=0\text{ and } L(\de x+\de u)=0$. Then $\de u$
is forced to belong to $\ker L$ and hence $D\pi(u,x) $ is not surjective. 

Conversely, if for every $x\in \pi^{-1}(u)$ the tangent spaces
$T_xS_1$ and $T_{x+u}S_2$  are not coplanar 
the matrix of $\left(\begin{array}{c}
Df_1(x) \\
Df_2(x+u)
\end{array}\right)$ is of maximal rank. 
Then, for every $\de u$ one can solve the linear system
 \begin{equation}
 \left\{
 \begin{array}{cl}\label{sec-lin2}
 Df_1(x)\cdot\de x&= 0\\
 Df_2(x+u)\cdot\de x&=-Df_2(x+u)\cdot \de u,
 \end{array} 
 \right. 
\end{equation}
and hence,  $u$ is a regular value of $\pi$.\bull

Here is the corollary we are interested in;
 non-coplanarity is a criterion for a translation flow to be of {\it  immediate tranversality} (Definition \ref{immediate}).

\begin{cor}{\bf (Non-coplanarity criterion)}\label{trans-cones}
 Let $S\subset \mathbb S^{n-1}$ be a smooth compact submanifold %of positive codimension 
with $C^1$ conic singularities\,\footnote{ The strata are $C^\infty$ but the local trivialization of the transverse conic structure is only $C^1$ at the vertex of the cone in each fiber.} in the unit $(n-1)$-sphere.
Let $C$ be the cone based on $S$ with the origin $O$ as a vertex. 
Then, there exists some  residual 
set $R\subset \vec \R^n$, actually an open and dense subset, 
such that  $u$ belonging to  $R$ is equivalent to each of the  
following properties:
\begin{enumerate}
\item The translated cone $C+u$ is transverse to $C$. % if and only if $u\in R$;
\item The translation flow generated by $u$ is a flow of immediate transversality to $C$.
\end{enumerate}.
\end{cor}
\proof We first show the two items are equivalent. 
Let  $(S_1,S_2)$ be a pair  of strata from the cone $C$  (that is, punctured cones based on strata in $S$.)
  If $S_1+u$ is not transverse to $S_2$ at $a$
then  $S_1+tu$ is not transverse to $S_2$ at $ta$; indeed, the tangent spaces to $S_1$ and $S_2$ are 
constant along each  generating line. Therefore,  non-transversality is preserved along a positive translation semi-flow. %at positive time.

In Proposition \ref{b1}, we have checked that the critical set of $\pi: Sec_{S_1,S_2}\to \vec \R^n$ is the set of  pairs
$(u,x)$ displaying coplanarity of $T_xS_1$ and $T_{x+u} S_2$.
So, $crit (\pi)$ is a cone; it is closed in $(\vec\R^n\setminus \{0\})\times S_1$ as usual for a critical set.  
One also controls its closure as 
we explain below.

It is easy to describe $\{0\}\times S_1$ as a part of the  closure of $crit (\pi)$ in $\vec \R^n\times S_1$.
%The completion of $crit (\pi)$ in $\vec \R^n\times S_1$ by $\{0\}\times S_1$ is easy. 
Indeed, for $(u_j,x_j)$ tending to $(0,x_0)$ with $x_0\in S_1$,
%approaching  $\{0\}\times S_1$, up to a subsequence, 
the ray $\R_+x_j$ goes
to the %some r
ray $\R_+x_0$. %with $x_0\in S_1$.
 A pair $(x,y)$ of points staying on the same ray makes $(y-x,x)\in crit(\pi)$. Therefore, the renormalized sequence
$(u_j/\Vert u_j\Vert,x_j/\Vert u_j\Vert)$ is asymptotic to this part of $crit(\pi)$
which is isomorphic to $\R_+\times S_1$.

Let us now consider the case of $(x_j)$ going to $x_0$ in another stratum $S_0$ of $C$; this stratum lies in the closure 
of $S_1$.
 Up to a subsequence, the sequence of 
tangent spaces $T_{x_j} S_1$ has a limit which contains $T_{x_0}S_0$; this is Whitney's condition A which holds since the singularities are $C^1$ conic. 
Let $\bar \pi: \vec\R^n\times\bar S_1\to \vec\R^n$ denote the extension of $\pi$ to the closure of its domain. 
If $(u_j, x_j)\in crit(\pi)$ for every $j$---that is some coplanarity---this condition A implies
%By the above mentioned property, the limit of 
$\lim_j (u_j, x_j)$ in $\vec\R^n\times S_0$ is a critical point of $\bar \pi$.

The cone  $crit(\bar\pi)$  is a subcone  of the  cone based on $S\times S$. Since $S$ is compact $crit(\bar\pi)$
has a compact base.
The same holds for the set of critical values in $\vec \R^n$. 
Then the set $R_{S_1,S_2}$ of regular values of 
$\bar\pi$ is open; moreover it is dense, as stated in Proposition \ref{b1}.
The desired $R$ is the finite intersection of $R_{S_1,S_2}$ over all pairs of strata. 
\bull

\begin{rien} {\bf Product family of cones.\footnote{ This generalizes to locally trivial bundles.}} \label{product-family}
{\rm This consists of the product $V\times(\B^{n-k}, Q)$ 
where $Q$ is a cone in $\B^{n-k}$ as in the previous corollary and $V$ is a \emph{compact} $k$-dimensional 
manifold. A \emph{translation flow} 
$(u^t)$ is generated by a smooth section $u: V\to V\times \vec \R^{n-k}$,
 that is a %family of 
 translation vector $u(p)$ in each fiber $\{p\}\times \B^{n-k}$, depending smoothly on $p$.
 The flow acts on the fiber over $p$ by the formula 
\begin{equation}
u^t(p, x)=\bigl(p, x+t u(p)\bigr).
\end{equation}
The germ of this flow is said to be  of \emph{immediate transversality}
% with respect to %V\times Q 
 to $V\times Q$ if $u^t(V\times Q)$ is transverse to 
$V\times Q$ for every small  positive $t$. 

Since $Q$ is a cone and only translations are involved, the flow $(u^t)$ is of immediate transversality 
if and only if $u^\theta(V\times Q)$ is transverse to $V\times Q$ for some $\theta >0$. 
Explicitely, this reads by saying that for every $p\in V$ one of the two following properties holds:
\begin{equation} \label{pitch}
\left\{
\begin{split}
&\text{ -- The translation }u^\theta(p)\text{ maps }\{p\}\times Q\text{ transversely to itself in }
\{p\}\times \B^{n-k}.\\
&\text{ -- For every hyperplane }H\text{ in }\B^{n-k}\text{ bitangent to }Q\text{ at some points }x\text{ and }x+u^\theta(p),\\ 
&\quad\text{ the operator }\p_V u^\theta\vert_{p}\text{ maps the tangent space  }T_{p}V\times\{0\}\text{ transversely to the}\\
& \quad\text{ codimension-one space  }T_pV\times H.  
\end{split}
\right.
\end{equation}
Property (\ref{pitch}) is open in the $C^1$ topology of sections (see Corollary \ref{trans-cones}).
Using the same idea as in Proposition \ref{transverse-family}, namely applying Sard's theorem to a 
finite dimensional family of sections of $V\times  \vec\R^{n-k}$ which is submersive on each fiber, one gets that immediate transversality is generic. More precisely, we have the following.
}
\end{rien}

\begin{prop} \label{cone-family} Regarding the  smooth sections $V\to V\times \vec\R^{n-k}$
as generators of (germs of) fiberwise translation flows on $V\times \B^{n-k}$, the set of those which generate immediate 
transversality to $V\times Q$ is open and dense in the $C^1$ topology of smooth sections. \emph{For short,
these sections are
 said  to be generic}.
Moreover, the following relative version holds: every germ of generic section along boundary $\p V$ extends to a generic section over $V$.
\end{prop}

\proof For the relative version, the given germ  extends arbitrarily to $\tilde\si:V\to V\times \vec\R^{n-k}$. 
Then, $\tilde\si$ has a generic approximation $\si$  which can be connected to $\tilde\si\vert\p V$
among the generic germs thanks to openness. \bull

The only remaining issue is to make coexist Proposition \ref{cone-family} and Corollary \ref{trans-cones}
when a stratum of a manifold with conic singularities enters the $n$-ball about a 0-stratum.
Here is the main concept related to this question.

\begin{rien} {\bf The reduced translation flow.} {\rm In the setting of Corollary \ref{trans-cones}, we consider a
$k$-dimensional stratum $S_k$ of the cone $C$, $k>0$, and a compact subdomain 
${\underline S}_{\,k}$ 
%homotopy equivalent to $S_k$ 
(Subsection  \ref{tubes} (2).) By definition of conic singularities, 
$C$ induces a conic bundle over ${\underline S}_{\,k}$. 
%We assume slighly more, (that holds under assumption (ATI), Affine Transverse Incidence (Definition \ref{affine}),) 
Namely,
 there exists a tube $N_k$, which is a \emph{trivial} $(n-k)$-disc bundle over ${\underline S}_{\,k}$ whose fibers $N_{k,x}, \ x\in {\underline S}_{\,k}$, are \emph{planar} in the unit ball $\B^n$.  The fibers $C\cap N_{k,x}$ are conic and  form a trivial cone subbundle of $N_k$.

Let $u$ be the generator of a (germ of) translation flow in $\B^n$. For every $x\in {\underline S}_{\,k}$ and every 
$y\in N_{k,x}$ one uses the splitting of the tangent space
\begin{equation}\label{split-reduce}
T_y\B^n= T_x{\underline S}_{\,k}\oplus T_yN_{k,x} .
\end{equation}
Here, the tangent space $T_x{\underline S}_{\,k}$ is carried to $y$ by parallelism with respect to the ambient affine structure of $\B^n$
and $N_{k,x}$ is thought of as spanning an $(n-k)$-dimensional affine subspace  in $\R^n$.
The splitting %(\ref{split-reduce}) 
decomposes the vector $u$ into horizontal and vertical components at $x$, that is:
\begin{equation}\label{b.5}
u= u_\frak h ^k(x) \oplus u_\frak v^k(y) 
\end{equation} 
with $ u_\frak h ^k(x)\in T_x{\underline S}_{\,k}$ and $u_\frak v^k(y)\in T_yN_{k,x}  $.
Note that this splitting is constant along the fiber $N_{k,x}$, that is, independent of $y$.
The vertical component $x\mapsto u_\frak v^k(x)$ is a section of $N_k$ which 
is termed the \emph{reduction of $u$ to $N_k$}.   
}
 \end{rien}

\begin{rien} \label{re-process}
{\bf Reducing process.} {\rm 
Let $\p \underline S_{\,k}$ denote the frontier of $\underline S_{\,k}$ in $ int(\B^n)$.  
Fix also an interior collar neighborhood 
$W_k$ of $\p \underline S_{\,k}$ in $ \underline S_{\,k}$
and let $E_k$ denote the part of $N_k$ over $W_k$. Without loss of generality
we may assume $E_k\subset \B^n$.
 Let $\mu:W_k\to [0,1]$ be a smooth function equal to 1 near 
$\p \underline S_{\,k}$ and 0 near the opposite face of  $W_k$. %; one may impose that $\mu$ vanishes on a connected subset of $W_k$.
 This $\mu$ is lifted to  $E_k$ as a constant function in each fiber $N_{k,x}$. The lifted $\mu$ is still noted $\mu$ and 
called a {\it balancing function}.

The {\it balanced reducing process} consists of replacing the constant vector field $u$ on $E_k$
by the vector field  
\begin{equation}
u_{\mu}^k(x):= \mu(x)u_{\frak h}^k(x) +u_{\frak v}^k(x).
\end{equation}
It is constant in each fiber $E_{k,x}$. Note that $u_\mu^k$ is equal to $u$
 in the part of $N_k$ over
 a small neighborhood of $\p \underline S_{\,k}$. %\cap \B^n$.
  Such a vector field also reads 
  $$u_\mu^k= \mu \,u+ (1-\mu)u_{\frak v}^k.
  $$ 
  This vector field is termed the {\it balanced  reduction} of $u$.
  }
\end{rien}
 \begin{center}
\begin{figure}[h]
\includegraphics[scale=.6]{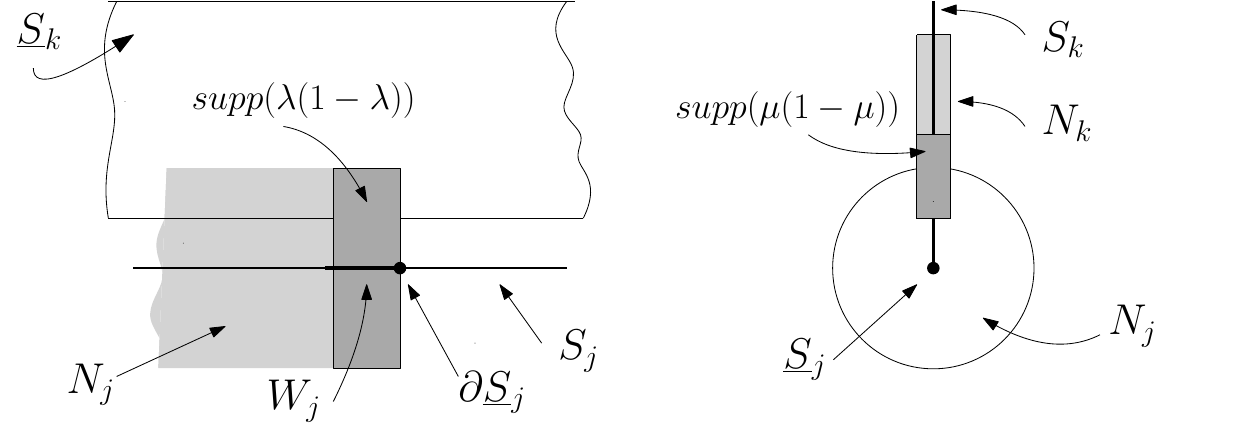}
 \caption{Two sectional views %drawings 
 of the tubes $N_j$ and $N_k$ in $\B^n$.}
 \label{assoc}
\end{figure}
\end{center}

\begin{rien}\label{assox}{\bf Skew associativity formula.} 
{\rm For $j<k$, let $S_j$ and $\underline S_{\,j}$ be a $j$-dimensional stratum 
 of the cone $C\subset \B^n$ and its compact subdomain; and let $\la: \underline S_{\,j}\to[0,1]$ be a balancing function
 for $S_j$.
 The position of $S_{\,j}$ with respect to $S_{\,k}$
 is specified in subsection \ref{tubes} (see Figure \ref{assoc}). If $N_{j,x}$ is a fiber of $N_j$, 
 with $x\in \underline S_{\,j}$,
   and $y$ is a point in
 $N_{j,x}\cap W_k$, %the assumption (ATI) implies 
 the fiber $N_{k,y}$ is an affine subspace of $N_{j,x}$. Then, 
 the reducing process with respect to stratum $S_k\cap N_{j,x}$ may be applied to
  the translation vector  $u^j_\frak v$ in the $(n-j)$-ball $N_{j,x}$. One gets the next fomula along 
the   fiber $N_{k,y}\subset N_{j,x}$:
 \begin{equation}\label{asso-formula}
u^k_\frak v= \left(u^j_\frak v\right)^k_\frak v \quad\text{and}\quad u^k_\frak h= u^j_\frak h+
\left( u^j_\frak h\right)^k_\frak h
\text{ that is,}
 \end{equation}  
$$u^k_\mu= \mu\la \, u^j_\frak h +\mu\left(u^j_\frak  h\right)^k_\frak h + \left(u^j_\frak v\right)^k_\frak v.
$$
Note these formulas hold regardless of the functions $\la$ and $\mu$. 
The subscript $\frak h$ has two different meanings: one stands for parallelism to $T_xS_j$ and the second one for parallelism in $N_{j,x}$ to 
$T_y\bigl(S_k\cap N_{j,x}\bigr)$.

There are analogous formulas associated with a sequence of strata $S_{j_1}, S_{j_2}, ..., S_{j_r}$ when each 
is in the closure of the next one.
}
 \end{rien}

 \begin{prop} \label{balanced}  In the setting of Corollary \ref{trans-cones} of a stratified cone $C\in \B^n$,
 it is assumed that the conic transverse structure
 to each stratum has a global trivialization.\footnote{ This condition is fulfilled in the case of \emph{simple} Morse-Smale 
 gradients of a Morse function (see Definition \ref{simple}.)} 
  If $u$ generates a flow 
 of immediate transversality to $C$ then we have:
 \begin{enumerate} 
 \item  The reduction $u^k_\frak v$ of $u$ to $N_k$ generates a flow of immediate transversality to $C\cap N_k$.
 \item The flow generated by the balanced reduction of $u$ to $N_k$
 is of immediate transversality to $C\cap E_k$.
 \end{enumerate}
   \end{prop}
 
 \proof The matter deals with bi-1-jets of $C$ (or pairs of tangent planes to $C$.) This allows one to linearize
 the considered vector field at any desired point without changing the problem.

On the linear disc bundle $N_k$ we have two \emph{linear} connections $h_0$ and $h_1$ (seen as plane distributions
complementary to the fibers): $h_0$ is parallel to $T_xS_k$ along the fiber $N_{k,x}$  for every $x\in S_k$; and $h_1$ 
is given by the assumed global 
 trivialization of $N_k$. The \emph{difference} between them, seen as a vertical deviation, is measured 
by a 1-form $\om$ on $S_k$ valued in the vector space of
linear endomorphisms of the vector bundle spanned by $N_k$.

By assumption, the vector $u$ generates a translation flow of immediate transversality to $C$.
Let  $\al$ be the minimum angle  between $T_y C$ and a hyperplane %$H$
 containing $T_{y+tu}C$ for every $y\in C$ and small positive $t$.
  The lowest bound of this angle is positive by assumption on $u$; it is independent
 of $t$ since $C$ is a cone  and it is a minimum since $C$ has a  compact base. \\

\nd {\sc Claim.} {\it If $h_0=h_1$, then the statement holds.}
 
 Indeed, by the above assumption the distribution  $h_0$ is tangent to $C$.  % (and hence to $S_k$.) 
 Then, transversality to $C$
 translates to the vertical component of the flow of $u$.
 %Only the vertical component $u^k_\frak v$ of $u$ (resp. its balanced reduction $u^k_\la$) contributes to the vertical component of $\p_x u$ (resp. $\p_x u^k_\la$.) 
 By an order-one Taylor expansion at $y\in C$, 
an ``infinitesimal contact'', %at $t=0$, 
namely coplanarity of  $D_y(u^k_\frak v)(T_yC)$ to $T_yC$,
 implies at most transversality to $C$  with an arbitrarily small angle for  some small $t>0$, 
 contradicting $\al>0$. This proves (1) in this setting.
  If (2) fails, it should fail infinitesimally %tangentially 
  which is impossible by (1).\bull
  
  Let $y\in  N_{k,x}$ and let $y+tu^k_\frak v (x)$ be the vertically displaced point for a small time $t$; 
  suppose both points belong to $C$.
  The planes $h_1(y)$ and $h_1(y+tu^k_\frak v (x))$ are both tangent to $C$ but could be not parallel anymore.
  Nevertheless, thanks to the 1-form $\om$ which measures the ``difference $h_1-h_0$'', one computes
  that the angle between $h_1(y)$ and $h_1(y+tu^k_\frak v (x))$, the latter being translated to $y$, is a $O(t)$.
  Therefore, if $t>0$ is sufficiently small, this angle is negligible with respect to $\al$. So, the reasoning
  for the claim still holds.
  
   \bull

\section{Basics on homotopical algebras}\label{appendix-Astr}

In this appendix we review the basic terminology and result of the theory of $\A$-algebras. We refer the reader to 
Kenji  Lef\`evre-Hasegawa's thesis \cite{LefHas} for a comprehensive treatment. However, here we use the sign 
convention introduced \cite{keller}.  

Here, $\kk$ is a unitary ring. 

\begin{defn}
An $A_p$-algebra is a $\kk$-module equipped with a collection of $\kk$-module maps 
$m_i:A^{\otimes i} \to A$, $1 \leq i\leq p$ , of degree $2-i$  satisfying  the identities
\begin{equation} \label{defAinf}
\sum_{j+k+l=i} (-1)^{j+kl}m_{j+l+1}(1^{\otimes j} \otimes m_k\otimes 1^{\otimes l})=0
\end{equation}
for all $p \geq i\geq 1$.  
\end{defn}

Similarly an $\A$-algebra is a graded $\kk$-module $A$ together with a collection of $\kk$-module maps $m_i:A^{\otimes i} \to A$, $i\geq 1$ , of degree $2-i$  such for all $p$ , $(A, \{m_i\}_{1\leq i\leq p}) $ is an $A_p$-algebra.

\begin{remarque}{\rm
According to the sign convention in  \cite{LefHas} one should put $(-1)^{jk+l}$ instead of $(-1)^{j+kl}$. It turns out that 
these two definitions are equivalent.  Indeed  if $(m_1, m_2,\cdots )$ is an $\A$-structure according to the sign 
convention of  \cite{LefHas}, then $(m_1, (-1)^{2\choose 2 }m_2,\cdots  (-1)^{i\choose 2 }m_i,\cdots )$ is an 
$\A$-structure by our sign convention. The sign conventions in \cite{LefHas} is justified  by the cobar construction.
  The signs in \cite{keller} correspond to that of the opposite algebra in \cite{LefHas}.}

\end{remarque}

Let $(A, \{m_i\}_{1\leq i\leq p}) $ and $(A', \{m'_i\}_{1\leq i\leq p}) $ be two $A_p$-algebras. An $A_p$-morphism 
 from  $(A, \{m_i\}_{1\leq i\leq p}) $ to $(A', \{m'_i\}_{1\leq i\leq p}) $ consists of a collection of maps 
 $ f_i: A^{\otimes i}\to  A'$, $1 \leq  i\leq p$, with the $|f_i|=1-i$ satisfying the conditions

\begin{equation}\label{condition—Amorphism}
 \sum_{j+k+l=i}  (-1)^{{j+kl} }f_{j+l+1}\ (1^{\otimes j} \otimes m_k\otimes  1^{\otimes l})=\sum_{k=1}^i \sum_{i_1+\cdots +i_k=i} (-1)^{\epsilon_{i_1, \cdots i_k}}m'_k( f_{i_1} \otimes \cdots \otimes   f_{i_k})\\
 \end{equation}
 
\nd where   {$\epsilon_{i_1, \cdots i_k} = \sum_{j=1}^k (k-j)(i_j-1) $.}

\begin{remarque}{\rm
	If we follow the sign convention of \cite{LefHas}, then equation of \ref{condition—Amorphism}  transforms into 
\begin{equation}\label{condition—Amorphism-bis}
 \sum_{j+k+l=i}  (-1)^{l+jk} f_{j+l+1}\ (1^{\otimes j} \otimes m_k\otimes  1^{\otimes l})=\sum_{k=1}^i \sum_{i_1+\cdots +i_k=i} (-1)^{\epsilon_{i_1, \cdots i_k}}m'_k( f_{i_1} \otimes \cdots \otimes   f_{i_k})\\
 \end{equation}
 
 \nd where  $\epsilon_{i_1, \cdots i_k} = \sum_{j} ((1-i_j) \sum_{1\leq  k\leq j} i_ k) $.}
\end{remarque}
If $(m_i)$ and $(f_i)$ satisfy the equation (\ref{condition—Amorphism}), then $(m_1, (-1)^{2\choose 2 }m_2,\cdots (-1)^{i\choose 2 }m_i,\cdots )$ and $$(f_1, (-1)^{2\choose 2 }f_2,\cdots  (-1)^{i\choose 2 }f_i,\cdots)$$ satisfies (\ref{condition—Amorphism-bis}).

A collection of $\kk$-module maps $f=\{f_i\}_{i\geq 1}: A^{\otimes i}  \to A'$ is said to be a morphism of $\A$-algebras if for all $p$, $\{f_i\}_{ 1\leq i\leq p}$ is a morphism of  $A_p$-algebras.

%In particular
An $\A$-morphism
  $f=\{f_i\}_{i\geq 1}$  is said to be a quasi-isomorphism if the cochain complex map $f_1$ is a quasi-isomorphism.

\begin{defn} Let $A$ and $A'$ be two $A_\infty$-algebras with  the corresponding differentials  $D$ and $D'$ on the  bar constructions $BA$ and $BA'$.
Suppose that $f=\{f_i\}, g=\{ g_i\}: A\to A'$ are two $A_\infty$-morphisms and $F$ and $G$ are the coalgebra morphisms corresponding to $f$ and $g$.
Then  a homotopy between $f$ and $g$ is a $(F,G)$-coderivation $H: BA\to BA'$ such that 
\begin{equation} \label{homo-A}
F-G= D'H -HD.
\end{equation}
\end{defn}

\begin{thm} {\rm (Prout\'e \cite{Prou}, see also \cite{LefHas}.)} We suppose that  $\kk$ is a field. Then we have:
\begin{enumerate}
\item For connected $A_\infty$-algebras, homotopy is  an equivalence relation (Theorem 4.27). 
\item A quasi-isomorphism of $A_\infty$-algebras is a homotopy equivalence (Theorem 4.24).
\end{enumerate}

\end{thm}

\end{document}